\documentclass[reqno, 11pt]{amsart}
\usepackage[top=3.5cm, bottom=3.5cm, left=2.5cm, right=2.5cm, headsep=0.75cm]{geometry}            
\usepackage{amsmath,amsrefs,amssymb}
\usepackage{pgfplots}
 \usepackage{tikz,float}
\pgfplotsset{compat = newest}

\usepackage{graphicx}
\usepackage[colorlinks=true, pdfborder={0 0 0}]{hyperref}
\usepackage{subcaption}
\usepackage[font=scriptsize]{caption}
\usepackage{mwe}
\usepackage{dutchcal}
\usepackage{soul}

\newcommand{\R}{\mathbb{R}}

\newtheorem{theorem}{Theorem}[section]
\theoremstyle{plain}

\newtheorem{corollary}{Corollary}[section]

\newtheorem{lemma}{Lemma}[section]

\newtheorem{proposition}{Proposition}[section]
\newtheorem{remark}{Remark}[section]

\usepackage{thmtools}

\newcommand{\hn}{\mathbb{H}^{n}}

\newcommand{\divop}{div}

\numberwithin{equation}{section}

\allowdisplaybreaks

\begin{document}

\title[GH stability and rigidity of HPW]{Gromov-Hausdorff Stability and Rigidity of manifolds via Heisenberg-Pauli-Weyl Uncertainty Principle}

\author[]{Mousomi Bhakta}
\address{Mousomi Bhakta, Department of Mathematics\\
Indian Institute of Science Education and Research Pune (IISER-Pune)\\
Dr Homi Bhabha Road, Pune-411008, India}
\email{mousomi@iiserpune.ac.in}
\author[]{Debdip Ganguly}
\address{Debdip Ganguly, Theoretical Statistics and Mathematics Unit, 
Indian Statistical Institute, Delhi Centre\\
S.J. Sansanwal Marg, New Delhi, Delhi 110016, India}
\email{debdipmath@gmail.com,  debdip@isid.ac.in}
\author[]{Debabrata Karmakar}
\address{Debabrata Karmakar, Tata Institute of Fundamental Research\\
Centre For Applicable Mathematics\\
Post Bag No 6503, GKVK Post Office, Sharada Nagar, Chikkabommsandra, Bangalore 560065, India}
\email{debabrata@tifrbng.res.in}

\date{\today}
\subjclass[2010]{Primary: 53C24, 53C23, 53C21, 26D10, 58J60, 49Q20}
\keywords{}


\begin{abstract}

 The classical Heisenberg-Pauli-Weyl (HPW) inequality on $\mathbb{R}^n$ asserts that
 $$
 \left(\int_{\mathbb{R}^n} \vert{}\nabla u\vert{}^2 dx\right) \left(\int_{\mathbb{R}^n} \vert{}x\vert{}^2 u^2 dx\right) \geq \frac{n^2}{4} \left(\int_{\mathbb{R}^n} u^2 dx\right)^2
 $$
 holds for every $u \in H^1(\mathbb{R}^n)$ with $\vert{}x\vert{}u \in L^2(\mathbb{R}^n)$, with equality holding if and only if $u$ is Gaussian: $u_\lambda(x) = e^{-\lambda \vert{}x\vert{}^2}$ ($\lambda > 0$). In the Riemannian setting, this inequality exhibits a strong rigidity phenomenon. On Cartan–Hadamard manifolds, extremizers exist only when the manifold is isometric to $\mathbb{R}^n$, while, among manifolds with non-negative Ricci curvature, validity of the Euclidean HPW inequality itself forces the manifold to be isometric to $\mathbb{R}^n$. This geometric discrepancy suggests the introduction of a suitable curvature-dependent correction in the HPW formulation: for Cartan–Hadamard manifolds, the correction appears as an additional curvature-dependent term, while in case $\mbox{Ric} \geq 0$, it is  encoded through the asymptotic volume ratio (AVR).

 In this article, we investigate the following geometric stability question: given a sequence of pointed Riemannian manifolds $(M_k, g_k, o_k)$ and appropriately normalized functions $u_k \in C_c^\infty(M_k)$ for which the associated HPW deficit vanishes as $k \to \infty$, does the sequence of manifolds converge to the corresponding model space in the pointed Gromov-Hausdorff topology?

 We show that for pinched Cartan-Hadamard manifolds with sectional curvature bounded above by $c < 0$, convergence to the model hyperbolic space $\mathbb{H}^n_c$ holds. In the setting of manifolds with non-negative Ricci curvature, convergence to Euclidean space $\mathbb{R}^n$ is established up to a metric rescaling factor governed by the moment term of the sequence. As a corollary, we deduce that for non-negatively Ricci-curved manifolds that are not isometric to Euclidean space, the known HPW inequality formulated in terms of the asymptotic volume ratio (AVR) is suboptimal. To address this, we introduce a curvature-corrected HPW inequality in this setting, analogous to Cartan-Hadamard manifolds.

 In addition,  both in the non-negatively curved manifolds and in the negatively curved manifolds we establish respective quantitative rigidity results by demonstrating that the HPW deficit, when evaluated at Gaussian profiles, explicitly controls an appropriately defined distance from the corresponding model space.


\end{abstract}

\maketitle

\setcounter{tocdepth}{1}
\tableofcontents

\medskip

\noindent
\section{Introduction}

Optimal geometric and functional inequalities, together with their optimizers and stability properties, form a fundamental bridge between geometric analysis, partial differential equations, and mathematical physics. Over the past half-century, a broad class of fundamental inequalities-including the isoperimetric, Sobolev, Hardy, and Poincar\'e inequalities-has been studied extensively, leading to a deep and well-developed theory.

From the perspective of mathematical physics, another fundamental inequality is the Heisenberg uncertainty principle. In quantum mechanics, this principle asserts that the position and momentum of a particle cannot be simultaneously determined with arbitrary precision (see \cite{H}). Rigorously formulated in the mathematical framework of quantum mechanics by Pauli and Weyl \cite{W}, the uncertainty principle expresses, in analytic terms, the impossibility of simultaneously localizing a function and its Fourier transform to an arbitrarily high degree.

In Euclidean space $\mathbb{R}^n$, an equivalent formulation of the classical Heisenberg–Pauli–Weyl uncertainty principle asserts that
$$\left(\int_{\mathbb{R}^n} |\nabla u|^2 dx\right) \left(\int_{\mathbb{R}^n} |x|^2 u^2 dx\right) \ge \frac{n^2}{4} \left(\int_{\mathbb{R}^n} u^2 dx\right)^2\quad $$
holds for every $u \in H^1(\mathbb{R}^n)$ satisfying $\vert{}x\vert{}u \in L^2(\mathbb{R}^n).$
The constant $n^2/4$ is optimal, and equality is attained precisely by the family of Gaussian functions $u_\lambda(x) = e^{-\lambda |x|^2}$ ($\lambda > 0$).

When one seeks to extend functional inequalities from the flat Euclidean setting to a curved Riemannian manifold $(M,g)$, the geometry of the underlying space plays a decisive role in determining their validity, sharpness, and extremal structure. In particular, curvature can substantially influence both the optimal constants and the existence and characterization of optimizers.

The uncertainty principle has been extensively investigated in a broad range of settings, including Euclidean spaces, homogeneous groups, Riemannian manifolds, and more general metric measure spaces. For an overview of its developments and applications, we refer the reader to the surveys of Folland and Sitaram \cite{FS} and Fefferman \cite{F}, as well as the works of Nahmod \cite{N}, Ciatti-Ricci-Sundari \cite{CRS}, and Okoudjou-Saloff Coste-Teplyaev \cite{OST}. This list is by no means exhaustive. 
In the Riemannian setting, various forms of the Heisenberg--Pauli--Weyl uncertainty principle have been studied on both compact and noncompact manifolds; see, for instance, \cites{KO1,KO2,Anderson1,Anderson2,Erb1,Erb2,K}. Related uncertainty principles have also been developed in more general settings, including measure spaces, groups of polynomial volume growth, graphs, and fractals; see \cites{CRS,OST, HanXu}. In the Euclidean setting, stability questions for the Heisenberg uncertainty principle have received considerable attention, with several results relating the deficit to the family of extremals; see, for example, \cites{CFLL,DLL,MF,MV}.

\medskip

Throughout this article, we let $(M,g)$ be an $n$-dimensional complete, non-compact and connected Riemannian manifold, with $n\geq 2$. We denote by $dV_g$ the canonical Riemannian volume element and by $\nabla_g$ the gradient with respect to the metric $g$. For linearly independent vectors $v,w\in T_xM$, we denote by
$K_g(v\wedge w)$
the sectional curvature of the two-plane $\mbox{span}\{v,w\}\subset T_xM$. We denote by $\mbox{Ric}_M$ the Ricci curvature, viewed as a $(0,2)$-tensor field on $M$.

Let $d$ denotes the Riemannian distance induced by the metric $g$. Fix a point $x_0\in M$ we denote by $d_{x_0}(x):=d(x,x_0),$ the distance function from $x_0$. Throughout this article, we consider two distinct geometric regimes:

\begin{itemize}
\item[(a)] \textbf{Non-negatively Ricci curved manifolds:} $(M,g)$ is a complete non-compact Riemannian manifold satisfying
$\mbox{Ric}_M \geq 0.$

\item[(b)] \textbf{Cartan--Hadamard manifolds:} $(M,g)$ is a complete, simply connected Riemannian manifold whose sectional curvature satisfies
$K_g \leq c<0.$

Throughout the article, the constant $c$ is fixed unless otherwise stated.
\end{itemize}

In  \cite{K} Krist\'aly established a remarkable rigidity dichotomy governed by curvature for the following HPW principle on $(M,g)$:
$$\left(\int_{M} |\nabla_g u|^2 dV_{g} \right) \left(\int_{M} d_{x_0}^2 u^2 dV_{g} \right) \ge \frac{n^2}{4} \left(\int_{M} u^2 dV_{g} \right)^2 \quad \forall\, u\in C^\infty_c(M).  \qquad \text{\bf{(HPW)}}_{x_0} $$

\medskip

\begin{itemize}
    \item Non-negatively curved manifolds ($\text{Ric}_M \ge 0$): Inequality $\text{\bf{(HPW)}}_{x_0} $ holds if and only if $(M,g)$ is isometric to $\mathbb{R}^n$.  

    \item Cartan–Hadamard manifolds ($K_g \le 0$): While $\text{\bf{(HPW)}}_{x_0} $ holds with the Euclidean sharp constant $n^2/4$ across all Cartan-Hadamard spaces, positive extremal functions exist if and only if $(M,g)$ is  isometric to $\mathbb{R}^n$. 
\end{itemize}

As a consequence of Krist\'{a}ly's result it follows that for non-positively curved manifolds such as the hyperbolic space $\mathbb{H}^n$,  $\text{\bf{(HPW)}}_{x_0} $ possesses no positive extremals. 

\medskip

\subsection{HPW on Cartan–Hadamard manifolds}

To account for the geometric defect induced by curvature, Krist\'aly introduced a curvature correction term in the HPW-inequality, leading to what is referred to as the \emph{quantitative HPW uncertainty principle}. To state the precise result, we first introduce some notations.
 For a Cartan–Hadamard manifold $(M,g)$ with sectional curvature $K_g \le c \le 0$, ($c$  constant), define the functions $\text{{\bf ct}}_c : (0,\infty) \to \mathbb{R}$  by

\begin{align}\label{c-rho}
\text{{\bf ct}}_c(\rho) =
\begin{cases}
\dfrac{1}{\rho}, & \text{if } c = 0, \\[6pt]
\sqrt{|c|}\,\coth\!\bigl(\sqrt{|c|}\,\rho\bigr), & \text{if } c < 0,
\end{cases}
\end{align}

and  $\text{{\bf D}}_c : [0,\infty) \to \mathbb{R}$  by

\begin{align}
\text{{\bf D}}_c(\rho) =
\begin{cases}
0, & \text{if } \rho = 0, \\[6pt]
\rho\,\text{{\bf ct}}_c(\rho) - 1, & \text{if } \rho > 0.
\end{cases}
\end{align}

In \cite[Theorem 3.1]{K}, Krist\'aly showed that the natural counterpart of the classical HPW uncertainty principle on non-positively curved manifolds, which preserves the Gaussian extremals, is the \emph{$c$-quantitative HPW uncertainty principle}. More precisely, if $(M,g)$ is an $n$-dimensional Cartan--Hadamard manifold whose sectional curvature satisfies $K_g\leq c\leq 0,$
then, for every $x_0\in M$ and every $u\in C^\infty_c(M)$, the following inequality holds:
\begin{align*}
\left(\int_{M} |\nabla_g u|^2 dV_{g} \right) \left(\int_{M} d_{x_0}^2 u^2 dV_{g} \right) \ge \frac{n^2}{4} \left(\int_{M} \left(1 + \frac{n-1}{n} \text{{\bf D}}_c(d_{x_0})\right) u^2 dV_{g} \right)^2. \qquad\text{\bf{(HPW)}}_{x_0}^c
\end{align*}

\medskip

{\rm
   The constant $\frac{n^2}{4}$ is optimal in $\text{\bf{(HPW)}}_{x_0}^{\,c}$ and moreover, when $M = \mathbb{H}_c^n, c<0$, the hyperbolic Gaussians
\begin{align*}
u_\alpha(x)=e^{-\alpha d_{x_0}^2(x)}, \qquad \alpha>0,
\end{align*}
emerge naturally as extremal functions; see \cite[Theorem~1.3]{DGLL}.}

In \cite[Remark 5.1]{KKPZ} it was also mentioned that
 on a general Cartan–Hadamard manifold $(M,g)$ with sectional curvature $K_g \le c <  0$, if an extremal of $\text{\bf{(HPW)}}_{x_0}^{\,c}$ exists  then $M$ is isometric to $\hn_c$. For the convenience of the reader, we have given a detailed proof of this result in the appendix. 
 
 In this article, we are interested in the question of quantitative stability of the manifolds, viewed through the lens of the HPW-principle. For a Cartan-Hadamard manifold $(M,g)$ with sectional curvature bounded above by $c <0,$ we denote the $\text{\bf{(HPW)}}_{x_0}^{\,c}$ deficit:
 
 \begin{align*}
I_M(u) := \left(\int_{M} |\nabla_g u|^2 dV_{g} \right) \left(\int_{M} d_{x_0}^2 u^2 dV_{g} \right) - \frac{n^2}{4} \left(\int_{M} \left(1 + \frac{n-1}{n} \text{{\bf D}}_c(d_{x_0})\right) u^2 dV_{g} \right)^2.
 \end{align*}

Suppose we have a sequence of (pointed) Cartan-Hadamard manifolds $(M_k, g_k, o_k)$ with sectional curvature bounded above by $c <0,$ and functions $u_k \in C_c^{\infty}(M_k)$ such that the associated HPW deficit tends to zero: $\lim_{k \rightarrow \infty} I_{M_k}(u_k) = 0$. We ask under what conditions we can conclude that the sequence $M_k$ converges to $\mathbb{H}^n_c$ in a suitable geometric sense, such as the pointed Gromov-Hausdorff topology.

To establish such a limit, we must impose hypotheses that exclude two major analytical and geometric obstructions:
\begin{itemize}
\item Non-concentration: We must prevent $u_k$
 from concentrating at the pole. Without this assumption, a vanishing deficit may simply reflect the local Euclidean structure near the pole rather than the global geometry of the manifold, since any sufficiently concentrated Gaussian is an approximate extremizer.

\item Curvature lower bound: We require a uniform sectional curvature lower bound, $K_{g_k} \ge c_0$. Without it, the sequence of manifolds could collapse or lack the non-collapsing volume bounds required by Gromov's Compactness Theorem, thereby preventing the extraction of a limit space entirely.
\end{itemize}

Our first main result is the following stability result.

\begin{theorem}\label{th:stability-4}
Let $(M_k, g_k, o_k)$ be a sequence of pointed $n$-dimensional  Cartan-Hadamard manifolds such that the sectional curvature is bounded from above by $c < 0$ and below by $c_0$. Let $\rho_k(x) = d_{g_k}(o_k, x)$ and let $u_k \in H^1(M_k, g_k)$ be a sequence of functions satisfying the normalization and the non-concentration condition: 
\begin{align*}
\Vert{}u_k\Vert{}_{L^2(M_k)} = 1, \ \ \ 0<C^{-1}\leq\Vert{}\rho_k u_k\Vert{}_{L^2(M_k)} \leq C < \infty.
\end{align*}
Assume that the Heisenberg-Pauli-Weyl functional deficit satisfies:
$$\lim_{k \to \infty}I_{M_k}(u_k)=0.$$
Then, upto a subsequence $(M_k, g_k, o_k)$ converges in the pointed Gromov-Hausdorff topology to $\hn_c$, the simply connected $n$-dimensional space form of constant sectional curvature $c$.
\end{theorem}

As a direct consequence of this stability result, we recover the rigidity case of the HPW inequality: if a manifold $(M,g)$ satisfying the hypotheses of the theorem admits a non-trivial HPW-extremizer, then $M$ is isometric to $\mathbb{H}^n_c$. Indeed, this follows immediately by setting $M_k = M$ and taking $u_k$ to be the normalized extremizer.

\begin{remark}
The assumption on the upper bound $\Vert{}\rho_k u_k\Vert{}_{L^2(M_k)} \leq C$ may at first seem artificial; however, it is in fact a necessary condition, as the theorem fails without it. For example, consider that $M$ is a $n$-dimensional warped product manifold, $M = \mathbb{R} \times_{\cosh(\sqrt{|c|}r)} \mathbb{H}_{c'}^{n-1}$, where $c^{\prime}<c<0$. Then the sectional curvature of $M$ lies in $[c',c]$. One can then construct a sequence $u_k$ concentrating farther and farther along  a radial geodesic ray for which the deficit tends to zero, while
$\|\rho  u_k\|_{L^2(M)}\rightarrow\infty.$
\end{remark}

Next we turn to a quantitative version of the stability result, or rather a quantitative rigidity result. Note that the deficit $I_M(u)$ introduced above is homogeneous of degree $4$ in $u$. In order to simplify computational complexity, we introduce the homogeneous degree $2$ deficit 
\begin{align}\label{I_M}
\mathcal{I}_{M}(u) := \sqrt{\left( \int_M |\nabla_g u|^2 \, dV_{g} \right)
\left( \int_M d_{x_0}^2 \, u^2 \, dV_{g} \right)}
-
\frac{n}{2}\left( \int_M \left( 1 + \frac{n-1}{n} \text{{\bf D}}_c(d_{x_0}) \right) u^2 \, dV_{g} \right),
\end{align}
which is non-negative. We define the deficit of $M$ from being $\hn_c$ in terms of the $\mbox{\bf{(HPW)}}_{x_0}^c$-deficit by
\begin{align}\label{eq:deficit}
\delta(M) := \mathcal{I}_M\left( e^{-\frac{1}{2}d_{x_0}^2}\right).
\end{align}
Note that if $\delta(M) = 0,$ then by rigidity result,  Theorem~\ref{rigidity thm} below, $M$ is isometric to the hyperbolic space $\hn_c$. Now we want to define an appropriate notion of a distances from a Cartan-Hadamard manifold of sectional curvature $\leq c$ to the model space $\hn_c.$ Once again, by the rigidity of constant sectional curvature space, we know that if a manifold has constant sectional curvature $c,$ then it is isometric to the hyperbolic space $\hn_c.$ Therefore, it makes sense to define distance as a norm of $(c - K_g)$ in some Lebesgue space. 

By Cartan-Hadamard Theorem, the exponential map $\exp_{x_0}: T_{x_0}M \to M$ is a smooth diffeomorphism from the tangent space $T_{x_0}M$ onto $M$. The global diffeomorphism of $\exp_{x_0}$ allows the definition of geodesic polar coordinates on the entire manifold. For $x \in M \setminus \{x_0\}$, we write $x = \exp_{x_0}(\rho\theta)$ where $ \rho = d(x_0, x)$ is the radial distance and $\theta \in \mathbb{S}^{n-1} \subset T_{x_0}M$ is a unit tangent vector. For any  $\theta\in\mathbb{S}^{n-1} \subset T_{x_0}M$, let $\gamma_\theta(.)$ be the unique unit-speed geodesic starting at $x_0$ in the direction $\theta$: 

$$\gamma_\theta(t) = \exp_{x_0}(t\theta), \quad t \geq 0.$$

Since $M$ is a Cartan-Hadamard manifold, geodesics never intersect (except at $x_0$) and cover the entire manifold. By Gauss lemma, the radial geodesics emanating from $x_0$ remain perpendicular to the geodesic spheres $S_\rho(x_0)$ centered at $x_0$. Consequently, $g\left(\frac{\partial}{\partial \rho}, \frac{\partial}{\partial \theta^i}\right) = 0.$ Furthermore, because $\gamma_\theta(\rho)$ is a unit-speed geodesic, $g\left(\frac{\partial}{\partial \rho}, \frac{\partial}{\partial \rho}\right) = 1$. In these polar coordinates $(\rho, \theta^1, \theta^2, \dots, \theta^{n-1})$, locally the Riemannian metric $g$ can be decomposed into a radial component and a spherical component:
\begin{align*}
g = d\rho^2 + g_{ij}(\rho, \theta) d\theta^i d\theta^j
\end{align*}
and the volume element is given by:\
\begin{align}\label{8-19-1}
dV_{g} = \sqrt{\det(g_{ij}(\rho, \theta))} \, d\rho \, d\theta^1 \cdots d\theta^{n-1} = J(\rho, \theta) \, d\rho \, d\Theta
\end{align}

where $d\Theta$ is the standard volume element on the unit sphere $\mathbb{S}^{n-1}\subset T_{x_0}M$ and $J(\rho, \theta)$ is the density function. On $\mathbb{S}^{n-1}$, let  the standard metric $h$ in the same local coordinates $(\theta^1, \dots, \theta^{n-1})$ be given by $h_{ij}(\theta) d\theta^i d\theta^j$. Then the standard volume element on the sphere is $d\Theta = \sqrt{\det(h_{ij}(\theta))} \, d\theta^1 \dots d\theta^{n-1}$
 and therefore $J(\rho, \theta) = \sqrt{\frac{\det(g_{ij}(\rho, \theta))}{\det(h_{ij}(\theta))}}.$

Since, for each $\theta \in \mathbb{S}^{n-1}$,  $\gamma_{\theta}$ is the unit speed radial geodesic in the direction of $\theta$, for every $t\in [0,\infty)$, using parallel transport along $\gamma_{\theta}$, we can consider $\big(\dot{\gamma}_{\theta}(t), (e_i(t))_{i=1}^{n-1}\big)$ as an orthonormal basis of $T_{\gamma_{\theta}(t)}M.$ Then we can define the sectional curvature $K_g(\dot{\gamma}_\theta(t) \wedge e_i))$ of the plane spanned by $\dot{\gamma}_{\theta}(t)$ and $e_i(t), 1 \leq i \leq n-1.$ For simplicity, we will omit the dependence of $\theta$ and $t$ from the notation of $(e_i)_{i=1}^{n-1}.$
  
  \medskip
  
  We then define the distance of a pinched Cartan-Hadamard $M$ from being a constant curvature manifold $\hn_c$ by
 \begin{align*}
 \text{dist}(M, \hn_c) :=  \int_{\mathbb{S}^{n-1}} \left( \int_0^\infty \left[ \sum_{i=1}^{n-1} |c - K_g(\dot{\gamma}_\theta(\rho) \wedge e_i)| \right] e^{-\rho^2} J(\rho, \theta) d\rho \right) d\Theta.
 \end{align*}
Note that since $M$ is pinched between $c_0 = -\kappa_0^2$ and $c= - \kappa^2,$ the sectional curvature $K$ is bounded, and the volume coefficient $J(\rho, \theta)$ can have at most exponential growth. Hence, the above distance is always finite. Also, note that $\text{dist}(M, \hn_c) = 0$ implies that $M$ is isometric to the hyperbolic space with constant sectional curvature $c.$  We prove the following rigidity estimate:
\begin{theorem}\label{th:stability}
Let $(M,g)$ be a pinched Cartan-Hadamard $n$-manifold with sectional curvature bounded above by $c  = -\kappa^2 <0$ and bounded below by $c_0 =-\kappa_0^2.$ Then there exists a constant $C := C(n,\kappa, \kappa_0) > 0$ such that 
\begin{align*}
\text{dist}(M, \hn_c)  \leq C 
\begin{cases}
\delta(M), \ \ \ \ \ \ \mbox{if} \ \ \ n \geq 3, \\
\delta(M)^{\frac{1}{2}}, \ \ \ \ \mbox{if} \ \ \ n = 2.
\end{cases}
\end{align*}
\end{theorem}

\begin{remark}
{\rm
For $n\geq 3$, the stability exponent $1$ appearing in the deficit term is optimal, in the sense that it cannot be replaced by any larger exponent (see Subsection~4.1). In contrast, when $n =2,$ the corresponding exponent is $1/2.$ At present, it is not known whether this exponent is optimal. The appearance of the exponent $1/2$ stems from the underlying proof, in particular technical in nature, and it remains an interesting open problem to determine the optimal exponent in dimension two.}
\end{remark}

\subsection{Manifold with non-negative Ricci}

As proved by Krist\'aly, the uncertainty principle $\mathbf{(HPW)}_{x_0}$ cannot hold on a general manifold with $\operatorname{Ric_M} \geq 0$. To derive an analogous HPW principle on a complete, non-compact $n$-dimensional Riemannian manifold $(M,g)$ with $\operatorname{Ric}_M \geq 0$, one must take into account the asymptotic volume ratio:
\begin{align}\label{8-5-2}
\theta := \text{AVR}(M) := \lim_{r \rightarrow \infty} \frac{\text{Vol}_g(B(x_0, r))}{\omega_n r^n}.
\end{align}
Since $M$ has non-negative Ricci curvature, the Bishop-Gromov volume comparison theorem implies that the function $r \mapsto \frac{\operatorname{Vol}_g(B(x_0, r))}{\omega_n r^n}$ is non-increasing, and therefore $\theta \leq 1$. If $\theta > 0$, we say that $M$ has Euclidean volume growth. Moreover, $\theta = 1$ if and only if $\operatorname{Vol}_g(B(x_0, r)) = \omega_n r^n$ for every $r > 0$. By the equality case of the Bishop-Gromov comparison theorem, this guarantees that $M$ is isometric to Euclidean space.

In \cite[Theorem 3.2]{DLL}, the authors demonstrated that the following HPW uncertainty principle holds: Fix a base point $x_0 \in M,$ then 
\begin{equation}\label{HPW-theta}
\mathrm{I}_M(u) := \left( \int_M d_{x_0}^2 u^2 \, dV_{g} \right) \left( \int_M \vert{}\nabla_g u\vert{}^2 \, dV_{g} \right) - \frac{n^2}{4} \theta^{2/n} \left( \int_M u^2 \, dV_{g} \right)^2 \geq 0,  \quad\forall\, u \in C_c^{\infty}(M).
\end{equation}

This result relies on sharp isoperimetric inequality of Brendle \cite{Brendle} (\cites{iso1,iso2} for low dimensions) and the  P\'olya-Sze\"go inequality for manifolds with non-negative Ricci curvature \cite{BK}. The approach in \cite{DLL} uses Euclidean symmetrization (see Section~5 for the definition). Given a non-negative function $u \in C_c^{\infty}(M)$, let $u^{\star}$ be its symmetric decreasing rearrangement defined on $\mathbb{R}^n$.  P\'olya-Sze\"go inequality By Balogh and Krist\'aly states
\begin{align*}
\|\nabla_g u\|_{L^2(M)} \geq AVR(M)^{\frac{1}{n}} \|\nabla u^{\star}\|_{L^2(\mathbb{R}^n)}.
\end{align*}

Consequently, the HPW deficit  $I_M(u)$ (associated with the inequality \eqref{HPW-theta}) is bounded below by $\operatorname{AVR}(M)^{2/n}I_{\mathbb{R}^n}(u^{\star})$, which is always non-negative. Furthermore, if equality is achieved by some function $u$, then equality must also hold in the  P\'olya-Sze\"go inequality \cite[Proposition 3.1]{BK}. This forces $\theta = 1$, meaning $M$ must be isometric to the Euclidean space $\mathbb{R}^n$.

A natural question to ask is whether the inequality \eqref{HPW-theta} is optimal. While failed attempt to construct an example to demonstrate its sharpness, we found that the inequality is, in fact, not sharp. This lack of sharpness emerges as a direct consequence of the stability results we initially set out to explore.

To simplify the notation, we define the following functionals:
\begin{align*}
E_{M}(u) := \int_{M} \vert{}\nabla u\vert{}^2 dV_{g}, \ \ V_{M}(u) := \int_{M} d_{M}(o, x)^2 u(x)^2 dV_{g}.
\end{align*}

Throughout this section, to distinguish the functional $I_M$ from the deficit associated with a Cartan-Hadamard manifold, we use the following notation. Since the choice of the base point \(o\) plays a crucial role, we make this dependence explicit in the definition of the deficit below. For a smooth function $u \in C^\infty(M)$ normalized such that $\Vert{}u\Vert{}_{L^2(M)} = 1$, we define the deficit functional as:
\begin{align}\label{deficit for Ric}
\delta(u; M, o) := \frac{4}{n^2 \operatorname{AVR}(M)^{2/n}} E_{M}(u) V_{M}(u) - 1 = \frac{4}{n^2 \operatorname{AVR}(M)^{2/n}} \mathrm{I}_M(u).
\end{align}

With this setup, we establish the following stability result for manifolds with non-negative Ricci curvature:
\begin{theorem}\label{th:stability-3}
Let $(M_k, g_k,o_k)$ be a sequence of pointed, complete, non-compact, connected, Riemannian $n$-manifolds with $\operatorname{Ric}_{M_k} \ge 0$. Let $u_k \in C_c^\infty(M_k)$ be non-negative functions normalized such that 
$\|u_k\|_{L^2(M_k)} = 1$. 
If the relative HPW deficits satisfy $$\lim_{k \to \infty} \delta(u_k; M_k, o_k) = 0,$$ then there exists a positive sequence $\{\lambda_k\}$ such that $\lim_{k \rightarrow \infty} \lambda_k^2 V_{M_k}(u_k) = 1,$ and up to a subsequence
 $(M_k, \lambda_k^2g_k, o_k)$ converges to $(\R^n, g_{\mbox{\tiny{Eucl}}}, o_{\R^n})$ in  pointed Gromov-Hausdorff topology. 
\end{theorem}

\begin{remark}
   In particular, if the non-concentration condition $V_{M_k}(u_k) \ge c_0 > 0$ holds, the scaling factors $\lambda_k$ cannot diverge to $+\infty$. As a result, the rescaling does not merely zoom into an infinitesimal neighborhood to capture tangent space, it preserves non-trivial global information about the underlying manifolds. Conversely, if $V_{M_k}(u_k) \rightarrow 0,$ one cannot infer, in general, that the rescaled manifolds converge to Euclidean space. While rescaling a fixed smooth manifold at a point yields its Euclidean tangent space, for a varying sequence of manifolds the corresponding rescaled spaces may converge to a limit with genuinely non-Euclidean geometry.
\end{remark}

As a direct consequence of this stability result, we recover the rigidity case of the HPW inequality: any manifold $(M,g)$ satisfying the hypotheses of the theorem that admits a non-trivial HPW-extremizer is isometric to $\mathbb{R}^n$ up to scaling the metric by a constant factor. Furthermore, the proof bypasses the rigidity case of the Pólya–Szegő inequality.

As another corollary of the above theorem we prove

\begin{corollary}\label{cor:best-const-1}
Let $(M,g)$ be a complete, non-compact, connected, Riemannian $n$-manifold with $\mbox{Ric}_M \geq 0.$ We further assume that $M$ is not isometric to the Euclidean space. Then 
\begin{align*}
\mbox{AVR}(M)^{\frac{2}{n}} < \inf_{u \in C_c^{\infty}(M), u \neq 0} \frac{\left( \int_M \vert{}\nabla u\vert{}^2 dV_g \right) \left( \int_M d(x,o)^2 u^2(x) dV_g \right)}{\frac{n^2}{4}\left( \int_M u^2 dV_g \right)^2} < 1.
\end{align*}
\end{corollary}

In this regard, it is natural to ask what the optimal HPW inequality is for manifolds with non-negative Ricci curvature. The proof of such a theorem generally relies on three main ingredients: integration by parts, the Laplacian comparison theorem, and the Cauchy-Schwarz inequality. Unfortunately, in the setting of non-negative Ricci curvature, the Laplacian comparison theorem yields a reverse inequality, rendering it inconclusive for deriving an optimal bound directly. Therefore, the most effective approach is to retain the Laplacian of the distance function explicitly, which holds globally in a distributional sense \cite{Mantegazza_et_al}.

We formulate a corresponding HPW inequality that incorporates a curvature correction term for non-negative Ricci curvature. This can be compared to Kristály's result in Cartan-Hadamard manifolds; however, while his correction term is positive, ours is naturally negative due to the underlying geometry. For the inequality we derive, its optimality is apparent, with the Gaussian naturally emerging as the extremizer for this HPW principle.

Furthermore, we can recover the $\theta$-dependent HPW principle from our result for functions that are radially symmetric and decreasing. These additional hypotheses are necessary because the Laplacian of the distance function possesses a singular part concentrated on the cut locus of the base point. Consequently, when attempting to extend the result, the curvature correction term cannot be globally controlled solely by imposing $L^2$-normalizations. We refer the reader to Section~6 for a thorough discussion of these details.

Our next main results is a quantitative rigidity results associated with $\theta$-HPW uncertainty principle. 

We define the deficit as before:

\begin{align}\label{deficit}
\delta(M) := \delta\left(\frac{u}{\|u\|_{L^2}};M,x_0 \right), \quad\mbox{where}\quad u(x)=e^{-\frac{1}{2}d_{x_0}(x)^2},
\end{align}
where the later $\delta$ is defined as in \eqref{deficit for Ric}.
Regarding the the distance, one needs to be a bit careful, as the norm involving Ricci curvature is not enough. A Ricci flat manifold may not be isometric to the Euclidean space (for example the flat torus in dimension $n \geq 4$), unless the dimension is 2 and 3. The most direct choice is the quantity involving the asymptotic volume ratio:

\begin{align}\label{distance}
\text{dist}(M, \mathbb{R}^n) := (1 - \theta^{1/n}), \ \theta =\operatorname{AVR}(M).
\end{align}
Here the distance is modulo an isometry: $\text{dist}(M, \R^n) = 0$ if and only if $M$ is isometric to $\mathbb{R}^n,$ just as in the deficit. Then we have the following quantitative rigidity result:

\begin{theorem}\label{th:stability-2}
For every $\theta_{*}>0$, there exists a constant $C=C(n,\theta_{*})>0$ such that for every complete, non-compact $n$-dimensional Riemannian manifold $(M,g)$ satisfying $\text{Ric}_M \geq 0$ and $\operatorname{AVR}(M)=\theta \geq \theta_{*},$  the following estimate holds:
\begin{align*}
\delta(M) \geq C \theta^{2-1/n} \mbox{dist}(M, \mathbb{R}^n),
\end{align*}
where $\delta(M)$ and $\text{dist}(M, \R^n)$ are given as in \eqref{deficit} and \eqref{distance}. 
\end{theorem}

The stability and almost-rigidity of geometric spaces under nearly sharp functional inequalities have been extensively investigated in the literature. Without attempting to be exhaustive, we mention a few works \cites{CH, Ledoux, Xia, doCarmoXia04, NobiliViolo22}, together with the references therein and relevant survey articles \cites{Nobili25}. We apologize for the inevitable incompleteness of this selective list.

The article is organized as follows: Section 2, Preliminaries, reviews the core Riemannian geometry tools and various results that will be used in subsequent sections. In Sections 3 and 4, we prove the stability results for negatively curved manifolds, namely Theorem~\ref{th:stability-4} and Theorem~\ref{th:stability}. Section 5 deals with complete, non-compact Riemannian manifolds with $Ric_M \geq 0$, where we prove Theorem~\ref{th:stability-3}, Theorem~\ref{th:stability-2}, and Corollary~\ref{cor:best-const-1}. In Section 6, we discuss the optimal HPW inequality for positively curved manifolds. Finally, Section 7 is an appendix that outlines the rigorous form of the HPW Cartan-Hadamard rigidity result.


\section{Preliminaries : Geometric frameworks and Preparatory Results}

Let $(M,g)$ be an $n$-dimensional Riemannian manifold. We denote the open geodesic ball centered at $x \in M$ with radius $r > 0$ by $B_x(r) = \{y \in M : d(x,y) < r\}$. The canonical volume element on $(M,g)$ is denoted by $dV_g$. For any smooth Riemannian manifold, small geodesic balls satisfy the Euclidean asymptotic limit:
\begin{align*}
\lim_{r \to 0^+} \frac{\operatorname{Vol}_g(B_x(r))}{\omega_n r^n} = 1 \quad \forall\, x \in M,
\end{align*}
where $\omega_n$ denotes the volume of the unit ball in $\mathbb{R}^n$. 

Let $u : M \to \mathbb{R}$ be a $C^1$ function. In a local coordinate neighborhood with coordinates $(x^i)$, the local components of the differential are $u_i = \frac{\partial u}{\partial x_i}$, and the gradient field $\nabla_g u$ has components $u^i = g^{ij} u_j$, where $g^{ij} = (g_{ij})^{-1}$. Setting $\vert{}g\vert{} := \det(g_{ij})$, the Laplace–Beltrami operator $\Delta_g u = \operatorname{div}_g(\nabla_g u)$ is expressed locally as:
\begin{align*}
\Delta_g u = \frac{1}{\sqrt{\vert{}g\vert{}}} \frac{\partial}{\partial x^i} \left( \sqrt{\vert{}g\vert{}}\, g^{ij} \frac{\partial u}{\partial x^j} \right).
\end{align*}
For functions $u, v \in C^2(M)$ with $v$ having compact support, the standard integration by parts formula holds:
\begin{align*}
\int_M v \Delta_g u \, dV_{g} = -\int_M \langle \nabla_g v, \nabla_g u \rangle_g \, dV_{g}.
\end{align*}

\subsection{Polar Coordinates} 
Fix $x_0 \in M$. Since the exponential map is a local diffeomorphism near $0 \in T_{x_0}M$, writing $v = \rho\theta$ with $\rho > 0$ and $\theta \in S_{x_0}M ( \ \mbox{the unit tangent bundle}) \simeq \mathbb{S}^{n-1}$ defines geodesic polar coordinates by $x = \exp_{x_0}(\rho\theta).$ 
In general, these coordinates are defined only up to the cut locus of $x_0$. However, if $M$ is a Cartan-Hadamard manifold, $\exp_{x_0}$ is a global diffeomorphism, so they are globally defined on $M \setminus \{x_0\}$.

In this coordinates,
\begin{align*}
g = d\rho^2 + g_{\rho} = d\rho^2 + g_{ij}(\rho,\theta)\, d\theta^i d\theta^j.
\end{align*}
and the Laplace-Beltrami operator decomposes into radial and spherical components:
\begin{align*}
\Delta_g = \frac{\partial^2}{\partial \rho^2} + m(\rho,\theta)\frac{\partial}{\partial \rho} + \Delta_{\partial B(x_0,\rho)},
\end{align*}
where $\Delta_{\partial B(x_0,\rho)}$ is the Laplacian on the geodesic sphere $\partial B(x_0,\rho)$, and $m(\rho,\theta) = \frac{\partial}{\partial \rho} \left( \ln \sqrt{\vert{}g\vert{}} \right)$ represents the mean curvature of $\partial B(x_0,\rho)$ in the radial direction. We recall a result from \cites{GW, G}(also see \cite[Lemma 4.5]{BGG}).
\begin{lemma}\cites{GW, G}\label{vol bound}
Let $M$ be a $n$ dimensional Cartan-Hadamard manifold and $K_{rad}$  denotes it's sectional curvature in the radial direction. If 
$$K_{rad}(x)\leq -\frac{\psi''}{\psi} \quad\forall\, x\in M,$$
where $\psi$ is positive $C^2$ function such that $\psi(0)=0=\psi''(0)$, $\psi'(0)=1$ then 
$$m(\rho,\theta)\geq(n-1)\frac{\psi'(\rho)}{\psi(\rho)} \quad\forall\, r>0,\, \theta\in\mathbb{S}^{n-1}.$$
\end{lemma}

\subsection{Comparison Geometry} Let $d_{x_0}(x) := d(x_0, x)$ denote the distance function from the pole and let $\mathbf{ct}_c(\cdot)$ be defined as in \eqref{c-rho}. 

\subsubsection{Hessian and Laplacian Comparison \cite[Theorem 11.11]{J Lee}, \cite[Theorem 6.4.3]{Peter}} 

If $(M,g)$ is a Cartan-Hadamard manifold with sectional curvature bounded above by $c \leq 0$, the following hold in the distributional sense:
\begin{align*}
\operatorname{Hess} \ d_{x_0} \ge \mathbf{ct}_c(d_{x_0}) g_{\rho}, \ \ \Delta_g d_{x_0} \ge (n-1)\mathbf{ct}_c(d_{x_0}).
\end{align*}

Equality in the Laplacian comparison holds if and only if $(M,g)$ is isometric to the constant curvature space form $\mathbb{H}^n_c$ if $c <0$ or $\mathbb{R}^n$ if $c = 0$. For Cartan-Hadamard manifolds, the second part was stated in \cite[Theorem 2.1]{KKPZ}. For completeness, and to make the argument self-contained, we provide a proof of this assertion in the Appendix as it is essential in the rigidity of the HPW-principle stated in the introduction.

Let $\gamma : [0, T) \to M$ be a unit-speed geodesic issuing from $x_0 = \gamma(0)$, so that $d_{x_0}(\gamma(t)) = t$ and $\vert{}\dot{\gamma}(t)\vert{} = 1$ for all $t > 0$. At $\gamma(t)$, the level set of the distance function is the geodesic sphere $S_t(x_0) = \{y \in M : d_{x_0}(y) = t\}$. Its tangent space $T_{\gamma(t)}S_t(x_0) = \dot{\gamma}(t)^\perp$ is the orthogonal complement of the radial direction $\operatorname{span}\{\dot{\gamma}(t)\}$.

Let $\pi_t : T_{\gamma(t)}M \to T_{\gamma(t)}S_t(x_0) \subset T_{\gamma(t)}M$ be the orthogonal projection, given by $\pi_t(X) = X - g(X, \dot{\gamma}(t)) \dot{\gamma}(t)$ for all $X \in T_{\gamma(t)}M.$ Along $\gamma(t)$, the Hessian comparison inequality seen as a $(1,1)$-tensor can be expressed in the following equivalent way: 
\begin{align*}
\mbox{Hess} \ d_{x_0}|_{\gamma(t)} \ge \mathbf{ct}_c(t) \, \pi_t.
\end{align*}

\subsubsection{Volume Comparison} The following comparision holds:

\medskip
\noindent
Bishop-G\"unther (Sectional Curvature Bounds) (\cite[Theorem III.4.2]{Chavel} and \cite[Theorem III.4.4]{Chavel}): If the sectional curvature $K$ satisfies $c_0 \leq K \leq c \leq 0$, then the volume of geodesic balls is bounded by the corresponding space forms:
\begin{align*}
\operatorname{Vol}_{\mathbb{H}^n_c}(B_{x}(r)) \leq \operatorname{Vol}_g(B_{x}(r)) \leq \operatorname{Vol}_{\mathbb{H}^n_{c_0}}(B_{x}(r)).
\end{align*}
Equality on either side for some $r$ implies the ball is isometric to a ball in the respective space form. 

\medskip
\noindent
Bishop-Gromov (Ricci Curvature Bound \cite[Theorem 11.19]{JLee}): If $M$ has non-negative Ricci curvature ($\operatorname{Ric_M} \geq 0$), the ratio $r \mapsto \frac{\operatorname{Vol}_g(B_x(r))}{r^n}$ is non-increasing. Consequently:
\begin{align*}
\operatorname{Vol}_g(B_x(r)) \leq \omega_n r^n \quad \forall\, x \in M, \ r > 0.
\end{align*}

Equality holds if and only if the sectional curvature is identically zero.

\subsection{Gromov-Hausdorff (GH) Convergence} 
\begin{itemize}
    \item[$\bullet$] Let $(X_k, d_k)$ and $(X, d)$ be compact metric spaces. We say that $(X_k, d_k)$ converges to $(X, d)$ in the Gromov-Hausdorff topology, written: $(X_k, d_k) \xrightarrow{\text{GH}} (X, d)$ if $d_{\text{GH}}(X_k, X) \to 0$ as $k \to \infty$. The Gromov-Hausdorff distance between two compact metric spaces $X$ and $Y$ is defined as:
$$d_{\text{GH}}(X, Y) = \inf_{Z, f, g} d_H^Z\big(f(X), g(Y)\big)$$
where the infimum is taken over all metric spaces $(Z, d_Z)$ and isometric embeddings $f: X \to Z$ and $g: Y \to Z$, and $d_H^Z$ denotes the Hausdorff distance in $Z$.

 \item[$\bullet$] Pointed Gromov-Hausdorff Convergence: For pointed proper metric spaces $(X_k, d_k, x_k)$ and $(X, d, x)$ (not necessarily compact), we say $(X_k, d_k, x_k) \xrightarrow{\text{pGH}} (X, d, x)$ if, for every radius $R > 0$, the closed balls $\overline{B}_R(x_k) \subset X_k$ converge to $\overline{B}_R(x) \subset X$ in the GH topology.
 
\item[$\bullet$] Pointed Measured Gromov-Hausdorff Convergence: If the spaces are equipped with locally finite Borel measures $\mu_k$ and $\mu$, we say: $(X_k, d_k, \mu_k, x_k) \xrightarrow{\text{pmGH}} (X, d, \mu, x)$ if $(X_k, d_k, x_k) \xrightarrow{\text{pGH}} (X, d, x)$ and, under corresponding isometric embeddings $f_k: X_k \to Z$ and $f: X \to Z$ into a common metric space $Z$, the pushforward measures converge weakly: $(f_k)_\# \mu_k \rightharpoonup f_\# \mu.$ For pointed Riemannian manifolds $(M_k,g_k,x_k)$, we write
\begin{align*}
    (M_k,g_k,x_k)\xrightarrow{pGH}(X,d,x)
\quad\text{or}\quad
(M_k,g_k,x_k,\operatorname{vol}_{g_k})\xrightarrow{pmGH}(X,d,x,\mu),
\end{align*}
where $d_{g_k}$ and $\operatorname{vol}_{g_k}$ are the distance and volume measure induced by $g_k$.

\item[$\bullet$] We will use several standard compactness results in this setting, and will provide the appropriate references as they arise. For a concise introduction and compactness results we refer to \cite[Chapter 27]{Villani} and \cite[Chapter 11]{Peter}, and for a broader treatment and perspective, we refer to Gromov's monograph \cite{Gromov}.
 \end{itemize}

\subsection{Rigidity of HPW inequality} We end this section with the following rigidity result of the Cartan-Hadamard manifolds.

\begin{proposition} \label{rigidity thm}
Let $(M,g)$ be a $n$-dimensional Cartan-Hadamard manifold of sectional curvature bounded above by $c \leq 0.$ Assume that the equality is achieved in $\mbox{\bf{(HPW)}}_{x_0}^c$. Then $(M,g)$ is isometric to $\hn_c$, the hyperbolic space with constant sectional curvature $c$.
\end{proposition}
This result is mentioned in a remark of \cite[Remark~5.1]{KKPZ}. For reader's convenience we present here a proof of the above result. 

\begin{proof}
Let $u$ be an extremal for
$\text{\bf{(HPW)}}_{x_0} ^c$.  Using the Laplace comparison, we have
\begin{align}\label{J23-1}
\int_M \Delta_g(d_{x_0}^2)\, u^2 \, dV_{g} 
&= 2 \int_M \bigl(1 + d_{x_0}\,\Delta_g d_{x_0}\bigr)\, u^2 \, dV_{g} \nonumber\\
&\ge 2 \int_M \bigl(1 + (n-1)\, d_{x_0}\,\text{{\bf ct}}_c(d_{x_0})\bigr)\, u^2 \, dV_{g} \nonumber\\
&= 2n \int_M \left(1 + \frac{n-1}{n} \text{{\bf D}}_c(d_{x_0}) \right) u^2 \, dV_{g}.
\end{align}

On the other hand, since $\Delta_g(d_{x_0}^2)$ is non-negative, by monotone convergence theorem, $$\int_M \Delta_g(d_{x_0}^2)\, u^2 \, dV_{g} = \lim_{r \rightarrow \infty} \int_{B(x_0,r)}\Delta_g(d_{x_0}^2)\, u^2 \, dV_{g},$$
along every sequence $r \rightarrow \infty$. In addition, since $\int_M d_{x_0}^2 u^2 dV_{g} < \infty,$ by tranforming the intergral in to polar co-ordinates, we conclude that  
$\liminf_{r \rightarrow \infty} \int_{\partial B(x_0,r)} d_{x_0}(x)u^2 dA = 0,$ where $dA$ is the induced surface area measure. Furthermore, using $\langle \nabla_g d_{x_0}, n\rangle_g = 1$ on $\partial B(x_0,r),$ where $n$ is the radial outward unit normal field, we get

\begin{align}\label{J23-2}
\int_M \Delta_g(d_{x_0}^2)\, u^2 \, dV_{g} &= \lim_{r \rightarrow \infty} \int_{B(x_0,r)}\Delta_g(d_{x_0}^2)\, u^2 \, dV_{g}\nonumber\\
&=  \lim_{r \rightarrow \infty} \left[-\int_{B(x_0,r)} \langle \nabla_g (u^2), \nabla_g (d_{x_0}^2) \rangle \, dV_{g} + 2 \int_{\partial B(x_0,r)} d_{x_0}(x)u^2 dA\right]\nonumber\\
&= -4 \int_M u\, d_{x_0} \, \langle \nabla_g u, \nabla_g d_{x_0} \rangle \, dV_{g}\nonumber\\
&\leq 4\left( \int_M |\nabla_g u|^2 \, dV_{g} \right)^\frac{1}{2}\left( \int_M d_{x_0}^2 \, u^2 \, dV_{g} \right)^\frac{1}{2}\nonumber\\
&= 2n\int_M \left(1 + \frac{n-1}{n} \text{{\bf D}}_c(d_{x_0}) \right) u^2 \, dV_{g}.
\end{align}
In the last line we have used that $u$ is an extremal of quantitative $\mbox{\bf{(HPW)}}_{x_0}^c$. Now comparing \eqref{J23-1} with \eqref{J23-2} implies inequality changed to equality everywhere. This in turn implies, equality holds in Laplace comparison 
\begin{align}\label{Laplace equality}
\Delta_g d_{x_0}(x) = (n-1)\,\sqrt{|c|}\,\coth \ \!\bigl(\sqrt{|c|}\,d_{x_0}(x)\bigr).
\end{align}
Therefore, the proof follows from the equality cases of Laplacian comparison result (see appendix Lemma \ref{Laplace rigidity lemma}). 
\end{proof}


\section{Stability of Cartan-Hadamard manifolds}

In this section we prove Theorem~\ref{th:stability-4}. We recall the $\mbox{\bf{(HPW)}}_{x_0}^c$-deficit is given by

\begin{align*}
I_{M}(u) = \left(\int_{M} \vert{}\nabla_g u\vert{}^2 \right) \left(\int_{M} \rho^2 u^2 \right) - \frac{n^2}{4} \left(\int_{M} \left(1 + \frac{n-1}{n} \text{{\bf D}}_c(\rho)\right) u^2 \right)^2\geq 0.
\end{align*}

{\bf Proof of Theorem~\ref{th:stability-4}:}
\begin{proof}
\noindent
{\bf Step 1.} 
In the proof of the uncertainity principle, there were two inequalities used: the Cauchy-Schwartz inequality, and the Laplace comparison inequality (after an integration by parts). 
 We decompose $I_M(u_k)$ as follows: $I_{M_k}(u_k) = \mathcal{E}_{\text{CS}}(u_k) + \mathcal{E}_{\text{Lap}}(u_k)$ where
\begin{align*}
\mathcal{E}_{\text{CS}}(u_k) &= \left(\int_{M_k} \vert{}\nabla_{g_k} u_k\vert{}^2 \right) \left(\int_{M_k} \rho_k^2 u_k^2 \right) - \left( \int_{M_k} u_k \rho_k \langle \nabla_{g_k} u_k, \nabla_{g_k} \rho_k \rangle \right)^2 \geq 0 \\
\mathcal{E}_{\text{Lap}}(u_k) &= \left( \int_{M_k} u_k \rho_k \langle \nabla_{g_k} u_k, \nabla_{g_k} \rho_k \rangle \right)^2 - \frac{n^2}{4} \left(\int_{M_k} \left(1 + \frac{n-1}{n} \text{{\bf D}}_c(\rho_k)\right) u_k^2 \right)^2 \geq 0,
\end{align*}
and $\rho_k(\cdot) = d_{g_k}(\cdot, o_k)$. Here $\mathcal{E}_{\text{CS}}(u_k) \ge 0$ follows from Cauchy-Schwarz inequality, and since $K_{g_k} \le c$, Laplacian comparison implies  $\Delta_{g_k} \rho_k \ge (n-1)\text{{\bf ct}}_c(\rho_k)$, and hence $\mathcal{E}_{\text{Lap}}(u_k) \ge 0$. Since $I_{M_k}(u_k) \to 0$, we deduce that both $\mathcal{E}_{\text{CS}}(u_k) \to 0$ and $\mathcal{E}_{\text{Lap}}(u_k) \to 0$.

\medskip

\noindent
{\bf Step 2:} For simplicity of notations, we denote
\begin{align*}
X_k := \Vert{}\nabla_{g_k} u_k\Vert{}_{L^2} , \ Y_k := \Vert{}\rho_k u_k\Vert{}_{L^2} , \ Z_k := - \int_{M_k} u_k \rho_k \langle \nabla_{g_k} u_k, \nabla_{g_k} \rho_k \rangle \, dV_{g_k} > 0.
\end{align*}
In the above line $Z_k>0$ follows using the arguments presented in the proof of Proposition~\ref{rigidity thm}. 
By step 1, setting $B_k := \int_{M_k} u_k^2 \left( 1 + \frac{n-1}{n}\text{{\bf D}}_c(\rho_k) \right) dV_{g_k}$, we get
 \begin{align*}
 \mathcal{E}_{\text{CS}}(u_k) = X_k^2 Y_k^2 - Z_k^2 \longrightarrow 0, \quad \mathcal{E}_{\text{Lap}}(u_k) = Z_k^2 - \frac{n^2}{4} B_k^2 \longrightarrow 0, 
 \end{align*}
 
 By hypothesis, $B_k \ge \int_{M_k} u_k^2 = 1$, and so $Z_k + \frac{n}{2} B_k \geq  \frac{n}{2},$ we get $Z_k - \frac{n}{2} B_k \rightarrow 0.$
 
 We set $\alpha_k := \frac{Z_k}{Y_k^2}.$ We claim that 
 $ \alpha_{\mbox{\tiny{min}}} \leq \alpha_k \leq \alpha_{\mbox{\tiny{max}}},$ for some $\alpha_{\mbox{\tiny{min}}}, \alpha_{\mbox{\tiny{max}}} >0$ independent of $k.$  In this step we use the hypothesis  $Y_k$ is uniformly bounded from above and below.
 Indeed, using $\text{{\bf D}}_c(\rho) \le \sqrt{\vert{}c\vert{}}\rho$, and Cauchy-Schwartz inequality, we get
\begin{align*}
B_k \le \int_{M_k} u_k^2 \left( 1 + \frac{n-1}{n}\sqrt{\vert{}c\vert{}}\rho_k \right) dV_{g_k} = 1 + \frac{n-1}{n}\sqrt{\vert{}c\vert{}} \int_{M_k} \rho_k u_k^2 \, dV_{g_k} \leq  1 + C_n \sqrt{\vert{}c\vert{}} Y_k,
\end{align*}
for some constant $C_n$ depending only on $n$. Now using $Z_k = \frac{n}{2}B_k + o(1)$, we can bound 
\begin{align*}
\alpha_k = \frac{Z_k}{Y_k^2} \leq  \frac{n}{2} \left( \frac{1 + C_n \sqrt{\vert{}c\vert{}} Y_k}{Y_k^2} \right) + o(1).
\end{align*}
Since $Y_k$ is bounded away from zero, we conclude $\alpha \leq \alpha_{\mbox{\tiny{max}}}.$ On the other hand, since $B_k \ge 1$, we have $Z_k \ge \frac{n}{2} + o(1)$ and therefore, 
\begin{align*}
\alpha_k = \frac{Z_k}{Y_k^2} \ge \frac{n/2 + o(1)}{Y_k^2}.
\end{align*}
Hence, the claim follows.

\medskip

\noindent
{\bf Step 3:} We set $V_k := \nabla_{g_k}u_k + \alpha_k \rho_k u_k \nabla_{g_k} \rho_k$ and claim that $\Vert{}V_k\Vert{}_{L^2(M_k, g_k)} \longrightarrow 0.$

 To prove it we compute 
 \begin{align*}
 \Vert{}V_k\Vert{}_{L^2}^2 &= \Vert{}\nabla_{g_k} u_k\Vert{}_{L^2}^2 + \alpha_k^2 \Vert{}\rho_k u_k\Vert{}_{L^2}^2 + 2\alpha_k \int_{M_k} u_k \rho_k \langle \nabla_{g_k} u_k, \nabla_{g_k} \rho_k \rangle dV_{g_k} \\
 &= X_k^2 + \alpha_k^2 Y_k^2 - 2\alpha_k Z_k\\
 &=X_k^2 + \left(\frac{Z_k}{Y_k^2}\right)^2 Y_k^2 - 2\left(\frac{Z_k}{Y_k^2}\right) Z_k\\
 &= X_k^2 + \frac{Z_k^2}{Y_k^2} - 2\frac{Z_k^2}{Y_k^2} = X_k^2 - \frac{Z_k^2}{Y_k^2}\\
 &= \frac{X_k^2 Y_k^2 - Z_k^2}{Y_k^2} = \frac{\mathcal{E}_{\text{CS}}(u_k)}{Y_k^2} \longrightarrow 0.
 \end{align*}

\medskip
\noindent
{\bf Step 4:} Define $f_k(x) := u_k(x) e^{\frac{\alpha_k}{2} \rho_k(x)^2}$, then 
\begin{align*}
\nabla_{g_k} f_k(x) = e^{\frac{\alpha_k}{2} \rho_k(x)^2} \Big( \nabla_{g_k} u_k(x) + \alpha_k \rho_k(x) u_k(x) \nabla_{g_k} \rho_k(x) \Big) = e^{\frac{\alpha_k}{2} \rho_k(x)^2} V_k(x).
\end{align*}

Fix any arbitrary radius $R > 0$ and consider the metric ball $B_R = B(o_k, R) \subset M_k$. On this ball, $\rho_k(x) \le R$. Since $\alpha_k \le \alpha_{max}$,    $e^{\frac{\alpha_k}{2} \rho_k(x)^2} \le e^{\alpha_{max} R^2 / 2} =: \Lambda(R) < \infty.$ Therefore,
\begin{align*}
\Vert{}\nabla f_k\Vert{}_{L^2(B_R)} \le \Lambda(R) \Vert{}V_k\Vert{}_{L^2(B_R)} \le \Lambda(R) \Vert{}V_k\Vert{}_{L^2(M_k)}  \longrightarrow 0 \quad \text{as } k \to \infty.
\end{align*}
By Poincar\'e inequality, 
\begin{align*}
\Vert{} f_k - C_k \Vert{}_{L^2(B_R)} \le C_P \Vert{}\nabla f_k\Vert{}_{L^2(B_R)} \longrightarrow 0 \quad \text{as } k \to \infty,
\end{align*}
where 
$C_k := \frac{1}{\text{Vol}_{g_k}(B_R)} \int_{B_R} f_k(x) \, dV_{g_k}(x).$ Therefore, 

\begin{align*}
u_k(x) = f_k(x)e^{-\frac{\alpha_k}{2} \rho_k(x)^2} = C_k e^{-\frac{\alpha_k}{2} \rho_k(x)^2} + \underbrace{ (f_k(x) - C_k) e^{-\frac{\alpha_k}{2} \rho_k(x)^2}}_{R_k(x)},
\end{align*}
where $R_k(x) \rightarrow 0$ and hence, $u_k=C_k e^{-\frac{\alpha_k}{2} \rho_k(x)^2}+o(1)$ in $L^2_{\mbox{\tiny{loc}}}$ as $k \rightarrow \infty.$ 

 Since sectional curvature of $M_k$ is unifomrly bounded from above by $c < 0$ and below by $c_0$, $\|u_k\|_{L^2} = 1,$ and $\alpha_{min}\leq\alpha_k\leq \alpha_{max}$, we claim that $C_k$ stays away from zero and infinity. This is a consequence of the fact that $M_k$ has at most exponential volume growth.

Indeed, for any fixed $R>0$,
\begin{align*}
1\geq\int_{B_{o_k}(R)}u_k^2dV_{g_k}=C_k^2 \int_{B_{o_k}(R)}e^{-\alpha_k \rho_k(x)^2}dV_{g_k}+o(1)\gtrsim C_k^2e^{-\alpha_{max}R^2}e^{\sqrt{|c|}R}.    
\end{align*}
Thus, $C_k$ is bounded from above. On the other hand, since $Y_k$ is bounded from above, we fix $R>0$ large enough such that
\begin{align*}
    \int_{M_k\setminus B_{o_k}(R)}u_k^2dV_{g_k}\leq\frac{1}{R^2}\int_{M_k\setminus B_{o_k}(R)}\rho_k^2u_k^2dV_{g_k}dV_{g_k}\leq \frac{Y_k^2}{R^2}<\frac{1}{2}.
\end{align*}

Therefore, 
$$\frac{1}{2}\leq\int_{B_{o_k}(R)}u_k^2dV_{g_k}=C_k^2 \int_{B_{o_k}(R)}e^{-\alpha_k \rho_k(x)^2}+o(1)\lesssim C_k^2e^{-\alpha_{min}R^2}e^{\sqrt{|c_0|}R}.$$
Hence, $C_k$ remains uniformly bounded away from $0$.

\medskip

\noindent
{\bf Step 5:} By step 2, $Z_k - \frac{n}{2} B_k \longrightarrow 0 \quad \text{as } k \to \infty.$ Expressing $Z_k$ in the following suitable form 
\begin{align*}
2 Z_k = -2 \int_{M_k} u_k \rho_k \langle \nabla_{g_k} u_k, \nabla_{g_k} \rho_k \rangle dV_{g_k} = \int_{M_k} u_k^2 \text{div}(\rho_k \nabla_{g_k} \rho_k) dV_{g_k} = \int_{M_k} u_k^2 (1 + \rho_k \Delta_{g_k} \rho_k) dV_{g_k},
\end{align*}
and using the definition of $B_k$ and $\text{{\bf D}}_c(\rho_k)$ we get
\begin{align*}
2 \left( Z_k - \frac{n}{2} B_k \right) &= \int_{M_k} u_k^2 (1 + \rho_k \Delta_{g_k} \rho_k) dV_{g_k} - \int_{M_k} n u_k^2 \left( 1 + \frac{n-1}{n} \text{{\bf D}}_c(\rho_k) \right) dV_{g_k},\\
&= \int_{M_k} u_k^2 \left[ \rho_k \Delta_{g_k} \rho_k - (n-1) - (n-1)\text{{\bf D}}_c(\rho_k) \right] dV_{g_k},\\
&= \int_{M_k} u_k^2 \rho_k \Big( \Delta_{g_k} \rho_k - (n-1)\text{{\bf ct}}_c(\rho_k) \Big) dV_{g_k}.
\end{align*}

Setting $\mathcal{L}_k(x) := \Delta_{g_k} \rho_k(x) - (n-1)\text{{\bf ct}}_c(\rho_k(x)) \geq 0,$ we conclude that $ \int_{M_k} u_k^2 \rho_k \mathcal{L}_k \, dV_{g_k} \longrightarrow 0$. Since $M_k$ is pinched, $\mathcal{L}_k$ are uniformly bounded. For $0  < r_1 < r_2 < \infty$, we denote the annulus by $A(r_1, r_2) = \{x \in M_k : 0 < r_1 \le \rho_k(x) \le r_2\}$. By step 4, $u_k = C_k e^{-\alpha_k \rho_k(x)^2 / 2} + o(1)$ in $L^2_{\mbox{\tiny{loc}}}$ and since both $\alpha_k$ and $C_k$ stays away from zero and infinity, we conclude
\begin{align*}
\lim_{k \to \infty} \int_{A(r_1, r_2)} \mathcal{L}_k(x) \, dV_{g_k} = 0.
\end{align*}

\medskip

\noindent
{\bf Step 6:} Let $\exp_{o_k}: T_{o_k}M_k \to M_k$ be the exponential map which is a global diffeomorphism. In this global polar coordinates $(r, \theta) \in (0, \infty) \times \mathbb{S}^{n-1}$ we write the metric and the volume element as:
\begin{align*}
g_k = dr^2 + g_{k,r}(\theta), \quad  dV_{g_k} = J_k(r, \theta) \,dr \,d \Theta,
\end{align*}
 where $J_k=\sqrt{\det(g_{k,r}( \theta))}$ is the density function as defined in \eqref{8-19-1} and $d\Theta$ is the standard volume element on the unit sphere in $T_{o_k}M_k$. Let $S_k(r) = \operatorname{Hess} \rho_k|_{\gamma_k(r)}: T_{\gamma_k(r)}M_k \to T_{\gamma_k(r)}M_k$ be the shape operator of the distance function along the geodesic $\gamma_k(r)$, i.e. $S_k(r)X=\nabla_{X}\partial_{\rho_k}.$  Then recall that the following relations hold: 
\begin{align}\label{deri J}
\partial_r \log J_k = \operatorname{tr} S_k = \Delta_{g_k}\rho_k.
\end{align}
On hyperbolic space $\hn_c$, the volume factor is $J_{\hn_c}(r) = \left(\frac{\sinh(\sqrt{|c|}r)}{\sqrt{|c|}}\right)^{n-1}$, which satisfies $$\partial_r \log J_{\hn_c} = (n-1)\text{\bf ct}_c(r).$$
Hence, $$\mathcal{L}_k=\partial_r \log J_k(r, \theta) - (n-1)\text{\bf ct}_c(r)= \partial_r \log \left( \frac{J_k(r, \theta)}{J_{\hn_c}(r)} \right)$$

 Let 
\begin{align*}
A_k(r) = \operatorname{Area}_{g_k}(\partial B_{g_k}(o_k, r)) = \int_{\mathbb{S}^{n-1}} J_k(r, \theta) \,d\Theta. 
\end{align*}

Differentiating and using \eqref{deri J} and $\mathcal{L}_k = \Delta_{g_k}\rho_k  - (n-1)\text{{\bf ct}}_c(\rho_k)$ yields:
\begin{align*}
A_k^{\prime}(r) = \int_{\mathbb{S}^{n-1}} \Delta\rho_k J_k(r,\theta) \,d\Theta = (n-1)\text{{\bf ct}}_c(r)A_k(r) + \int_{\partial B_r} \mathcal{L}_k \,dA_k.
\end{align*}
 Dividing by the model area $A_c(r) = \omega_{n-1} \left(\frac{\sinh(\sqrt{|c|}r)}{\sqrt{|c|}}\right)^{n-1}$, we obtain
\begin{align*}
\frac{d}{dr} \left( \frac{A_k(r)}{A_c(r)} \right) = \frac{1}{A_c(r)} \int_{\partial B_r} \mathcal{L}_k \,dA_k.
\end{align*}
For every fixed interval $0 < r_1 < r_2$, integrating the last identity we get 
\begin{align*}
\frac{A_k(r_2)}{A_c(r_2)} - \frac{A_k(r_1)}{A_c(r_1)} = \int_{r_1}^{r_2} \frac{1}{A_c(r)} \left( \int_{\partial B_r} \mathcal{L}_k \,dA_k \right) dr.
\end{align*}
Since $A_c(r)$ is bounded away from zero on $[r_1, r_2]$, by step 5, the right-hand converge to $0$ as $k \to \infty$. In fact we have local uniform convergence:
\begin{align} \label{uniform area conv}
\lim_{k \rightarrow \infty}\sup_{s_1,s_2 \in [r_1, r_2]} \left| \frac{A_k(s_2)}{A_c(s_2)} - \frac{A_k(s_1)}{A_c(s_1)} \right| = 0, \ \ \ \mbox{for every} \ 0 < r_1 < r_2 < \infty.
\end{align}
By volume comparison we know that for every $k$ and every $r > 0$, $A_c(r) \le A_k(r) \le A_{c_0}(r)$ and hence
\begin{align} \label{local unif conv}
0 \le \sup_k \left( \frac{A_k(r)}{A_c(r)} - 1 \right) \le \frac{A_{c_0}(r)}{A_c(r)} - 1 \rightarrow 0 \ \mbox{as} \ r \rightarrow 0.
\end{align}
Hence $\lim_{r \downarrow 0} \sup_k \left\vert{} \frac{A_k(r)}{A_c(r)} - 1 \right\vert{} = 0.$
Therefore \eqref{uniform area conv} and \eqref{local unif conv} gives $\frac{A_k(r)}{A_c(r)} \longrightarrow 1$ uniformly on compact subsets of $(0, \infty).$ As a result, since $ \mbox{Vol}_{g_k}(B(o_k, R)) = \int_0^R A_k(r) \,dr$, we get the volume convergence: 
\begin{align}\label{8-20-1}
\mbox{Vol}_{g_k}(B(o_k, R)) \longrightarrow \int_0^R A_c(r) \,dr = V_c(R),
\end{align}
where $V_c(R)$ is the volume of the corresponding ball in $\mathbb{H}^n_c$. \medskip

\noindent
{\bf Final step:} Since the sectional curvatures of $M_k$ are uniformly bounded above and below by two negative constants, the Cheeger–Gromov compactness theorem implies that, after passing to a subsequence, $(M_k, g_k, o_k) \longrightarrow (X_\infty, d_\infty, o_\infty)$
in the pointed  Gromov–Hausdorff topology, where $X_\infty$ is a $\text{CAT}(c)$ space. In fact, by \cite[Theorem 11.4.7]{Peter} and \cite[Proposition 5.14]{BH}, the convergence is $C^{1,\alpha}$ for every $\alpha\in(0,1)$, and $X_\infty$ carries the structure of a $C^{2,\alpha}$ Hadamard manifold endowed with a $C^{1,\alpha}$ Riemannian metric. Consequently, the corresponding volume measures converge locally, and hence, for every $R>0$,
\begin{align*}
\mathcal{H}^n\bigl(B_{X_\infty}(o_\infty,R)\bigr)
=\lim_{k\to\infty}\operatorname{Vol}{g_k}\bigl(B{g_k}(o_k,R)\bigr)
=V_c(R),
\end{align*}
where the last equality follows from \eqref{8-20-1} and $\mathcal{H}^n$ denotes the $n$-dimensional Hausdorff measure in $X_{\infty}$. By \cite[Proposition 6.1]{Nagano}, equality in the volume comparison theorem implies that $B_{X_\infty}(o_\infty,R)$ is isometric to the ball of radius $R$ in the hyperbolic space $\mathbb{H}_c^n$. 
Since $R>0$ is arbitrary and $X_{\infty}$ is simply connected, it follows that $X_\infty$ is globally isometric to $\mathbb{H}_c^n$. This completes the proof.

\end{proof}

\section{Quantitative rigidity in Cartan-Hadamard manifolds}

In this section we prove quantitative rigidity result, i.e., Theorem~\ref{th:stability}.
\begin{proof}
We fix a unit speed radial geodesic $\gamma$ starting at $x_0$ in the direction of $\theta$ (we keep the dependence on $\theta$ implicit). As before, let $\rho(x) = d_{x_0}(x)$  and $S(t) := \text{Hess}(\rho)|_{\gamma(t)} : T_{\gamma(t)}M \to T_{\gamma(t)}M$ be the shape operator of the distance function along the geodesic $\gamma(t)$, i.e. $S(t)X=\nabla_{X}\partial_{\rho}.$ The evolution of $S$ is governed by the Riccati equation:

\begin{equation}
    D_tS(t) + S(t)^2 + R_{\dot{\gamma}(t)} = 0,
\end{equation}
where  $R_{\dot{\gamma}(t)}X := R\big(X, \dot{\gamma}(t)\big)\dot{\gamma}(t)$ is the Riemann curvature endomorphism, and $D_t=\nabla_{\dot{\gamma}(t)}$ is the covariant derivative along $\gamma.$
The trace of $S(t)$ is the Laplacian of the distance function: $\text{tr}(S) = \Delta_g \rho|_{\gamma(t)}$.

Let $\pi_{t} : T_{\gamma(t)}M \to T_{\gamma(t)}M$ be  the orthogonal projection onto the tangent space of the level set of $\rho$ (onto the orthogonal complement of $\mbox{span} \ \partial_{\rho}|_{\gamma(t)}$). We define $S_c(t) = \text{{\bf ct}}_c(t)\pi_t$ and 
\begin{align*}
\mathcal{D}(t) := S(t) - S_c(t) = S(t) - \text{{\bf ct}}_c(t)\pi_t.
\end{align*}

Since $\dot{\gamma}$ is parallel, $D_t \dot{\gamma} = 0,$
therefore the orthogonal splitting
$T_{\gamma(t)}M=(\mbox{span} \ \partial_{\rho}|_{\gamma(t)})\oplus(\mbox{span} \ \partial_{\rho}|_{\gamma(t)})^\perp$
is parallel along $\gamma$. Consequently, $D_t \pi_t = 0.$
Moreover, as $S(t)$ annihilates the radial vector field, $S(t)$ and $S_c(t)$ commute.
Therefore the evolution for $\mathcal{D}$ is governed by
\begin{align*}
D_t \mathcal{D}(t) + \big(S(t)+S_c(t)\big)\mathcal{D}(t) = c \pi_t - R_{\dot{\gamma}(t)},
\end{align*}
where we used $\text{{\bf ct}}^{\prime}_c(t) + \text{{\bf ct}}_c(t)^2 = -c.$
Now substituting $S(t) = \mathcal{D}(t) + \text{{\bf ct}}_c(t)\pi_t$ into the above equation yields
\begin{align*}
D_t\mathcal{D}(t) + 2\text{{\bf ct}}_c(t)\pi_t\mathcal{D}(t) + \mathcal{D}^2(t) = c \pi_t - R_{\dot{\gamma}(t)}.
\end{align*}

By definition $\mathcal{D}(t)\partial_{\rho}=S(t)\partial_{\rho}-S_c(t)\partial_{\rho}=\nabla_{\partial_{\rho}}\partial_{\rho}-\text{{\bf ct}}_c(t)\pi_{t}\partial_{\rho}=0$ at $\gamma(t)$ and $\mathcal{D}(t)v\in (\mbox{span} \ \partial_{\rho}|_{\gamma(t)})^\perp$ if $v\in (\mbox{span} \ \partial_{\rho}|_{\gamma(t)})^\perp$, we have $\pi_tD(t)v=D(t)v$. Hence, 
$\pi_t\mathcal{D}(t)v=\mathcal{D}(t)v$ for all $ v\in T_{\gamma(t)}M$ and therefore 

\begin{align*}
D_t\mathcal{D}(t) + 2\text{{\bf ct}}_c(t)\mathcal{D}(t) + \mathcal{D}^2(t) = c \pi_t - R_{\dot{\gamma}(t)}.
\end{align*}

Now, we take the trace of this equation. Let $y(t) := \text{tr}(\mathcal{D}(t)) = \Delta_g \rho - (n-1)\text{{\bf ct}}_c(t)$ (as $\gamma$ is a unit speed geodesic). Since the trace of the curvature difference $(c\pi_t - R_{\dot{\gamma}(t)})$ over the orthogonal complement of $\dot{\gamma}(t)$ is $\sum_{i=1}^{n-1} (c - K(\dot{\gamma}(t) \wedge e_i))$, we get the scalar ODE:
\begin{align*}
y'(t) + 2\text{{\bf ct}}_c(t)y(t) + \text{tr}(\mathcal{D}(t)^2) = \sum_{i=1}^{n-1} (c - K(\dot{\gamma}(t) \wedge e_i)).
\end{align*}

Let the curvature difference along the radial geodesics be $E(\rho)$. Then  
\begin{align*}
E(\rho) = \sum_{i=1}^{n-1} |c - K(\dot{\gamma}(\rho) \wedge e_i)|.
\end{align*}

Thus from our Riccati ODE along the radial geodesics:
\begin{align}\label{E and y}
E(\rho) = y'(\rho) + 2\text{{\bf ct}}_c(\rho)y(\rho) + \text{tr}(\mathcal{D}^2).
\end{align}

Note that by definition 
\begin{align*}
\mbox{dist}(M, \hn_c)&=
\int_{\mathbb{S}^{n-1}} \left( \int_0^\infty \left[ \sum_{i=1}^{n-1} |c - K(\dot{\gamma}_\theta(\rho) \wedge e_i)| \right] e^{-\rho^2} J(\rho, \theta) d\rho \right) d\Theta,\\
&=\int_M E(\rho)e^{-\rho^2} dV_{g},
\end{align*}
where $\gamma_{\theta}$ is the radial geodesic starting at $x_0$ in the direction of $\theta$ and $J(\rho, \theta)$ is the Jacobian. 

Using \eqref{E and y} we get
\begin{align*}
\int_M E(\rho)e^{-\rho^2} dV_{g} 
= \int_M \left[ y'(\rho) + 2\text{{\bf ct}}_c(\rho)y(\rho) + \text{tr}(\mathcal{D}^2(\rho)) \right] e^{-\rho^2} dV_{g}.
\end{align*}
Consequently, 
\begin{align}\label{30-7-1}
\mbox{dist}(M, \hn_c)=\int_M \left[ y'(\rho) + 2\text{{\bf ct}}_c(\rho)y(\rho) + \text{tr}(\mathcal{D}^2(\rho)) \right] e^{-\rho^2} dV_{g}.
\end{align}

Now we simplify RHS of \eqref{30-7-1}. The term $y'(\rho)$ is the radial derivative, which we can rewrite as $\langle \nabla_g y, \nabla_g \rho \rangle_g$. Now using integration by parts   over $M$, we have
\begin{equation}\label{30-7-2}
    \int_M y'(\rho) e^{-\rho^2} dV_{g} = \int_M \langle \nabla_g y, \nabla_g \rho \rangle_g e^{-\rho^2} dV_{g} = - \int_M y(\rho) \text{div}(e^{-\rho^2} \nabla_g \rho) dV_{g},
\end{equation}
where we have used the fact that boundary terms at infinity vanish due to the fast decay of the Gaussian weight $e^{-\rho^2}$ against the at-most-exponential volume growth of $M$. Expanding the divergence term we get
\begin{align*}
\text{div}(e^{-\rho^2} \nabla_g \rho) = e^{-\rho^2} \Delta_g \rho + \langle \nabla_g(e^{-\rho^2}), \nabla_g \rho \rangle = e^{-\rho^2} \Delta_g \rho - 2\rho e^{-\rho^2}.
\end{align*}

Substituting the definition of $ y(\rho)$ i.e.,
$y(\rho) =\Delta_g \rho -(n-1)\text{{\bf ct}}_c(\rho)$, into the above relation yields

$$\text{div}(e^{-\rho^2} \nabla_g \rho) = e^{-\rho^2} \left[ y(\rho) + (n-1)\text{{\bf ct}}_c(\rho) - 2\rho \right].$$

Now plugging this back into \eqref{30-7-2} yields
\begin{equation}\label{30-7-3}
    \int_M y'(\rho) e^{-\rho^2} dV_{g} = \int_M y(\rho) \left[ -y(\rho) - (n-1)\text{{\bf ct}}_c(\rho) + 2\rho \right] e^{-\rho^2} dV_{g}.
\end{equation}
Substituting \eqref{30-7-3} into \eqref{30-7-1}, we obtain
\begin{equation}\label{30-7-4}
   \mbox{dist}(M, \hn_c) = \int_M \left[ -y^2(\rho) + \text{tr}(\mathcal{D}^2(\rho)) + 2\rho y(\rho) - (n-3)y\text{{\bf ct}}_c(\rho) \right] e^{-\rho^2} dV_{g}. 
\end{equation}
By Hessain comparison it follows that $\mathcal{D}\geq 0$. Consequently, all the eigenvalues of $\mathcal{D}$ are nonnegative, and hence $\mbox{tr}(\mathcal{D}^2)\leq (\mbox{tr}(\mathcal{D}))^2=y^2$. Therefore,
\eqref{30-7-4} reduces to

\begin{align} \label{error bound}
 \mbox{dist}(M, \hn_c) \leq \int_M y(\rho) \Big( 2\rho - (n-3)\text{{\bf ct}}_c(\rho) \Big) e^{-\rho^2} dV_{g}.
\end{align}

Now we recall the notion of deficit from \eqref{eq:deficit}:  $$\delta(M) = \mathcal{I}_{M}\left(e^{-\frac{1}{2}\rho^2}\right),$$
where $\mathcal{I}_M(\cdot)$ is as given in \eqref{I_M}.
To evaluate $\mathcal{I}_{M}\left(e^{-\frac{1}{2}\rho^2}\right)$, we first observe that 
\begin{align*}
\int_M|\nabla_g (e^{-\frac{1}{2}\rho^2})|^2dV_{g}=\int_M \rho^2e^{-\rho^2}\langle \nabla_g\rho, \nabla_g\rho\rangle_g dV_{g} =\int_M \rho^2e^{-\rho^2}dV_{g}.
\end{align*}

Plugging this integral into the definition of $I_{M}\left(e^{-\frac{1}{2}\rho^2}\right)$ yields
\begin{align}\label{30-7-5}
\delta(M)=\int_M \rho^2e^{-\rho^2}dV_{g}-\frac{1}{2}\int_M \bigl(1 + (n-1)\, \rho\,\text{{\bf ct}}_c(\rho)\bigr)\, e^{-\rho^2} \, dV_{g}.
\end{align}
Set, $X=\rho e^{-\rho^2}\nabla_{g}\rho$. Then using $\langle \nabla_g\rho, \nabla_g\rho\rangle_g=1$, we obtain
\begin{align*}
0=\int_M\mbox{div}_gX\, dV_{g}=\int_M e^{-\rho^2}dV_{g}-2\int_M \rho^2e^{-\rho^2}dV_{g}+\int_M\rho e^{-\rho^2}\Delta_g \rho\, dV_{g},
\end{align*}

where the first equality follows using integration by parts, since $e^{-\rho^2}$ decays faster compared to the volume element. Therefore, 
\begin{align*}
\int_M \rho^2 e^{-\rho^2}dV_{g}=\frac{1}{2}\int_M e^{-\rho^2}dV_{g}+\frac{1}{2}\int_M\rho e^{-\rho^2}\Delta_g g\, dV_{g}.
\end{align*}

Plugging this integral back into \eqref{30-7-5} yields
\begin{align}\label{30-7-6}
    \delta(M)&=\frac{1}{2}\int_M\big[\Delta_g\rho-(n-1) \text{{\bf ct}}_c(\rho)\big]\rho e^{-\rho^2} \, dV_{g}\nonumber\\
    & = \frac{1}{2}\int_M \rho y(\rho) e^{-\rho^2} dV_{g}. 
\end{align}
Comparing \eqref{30-7-6} with \eqref{error bound}, we obtain
\begin{equation}\label{30-7-8}
    \mbox{dist}(M, \hn_c) \leq 4\delta(M)- (n-3)\int_M y(\rho)\text{{\bf ct}}_c(\rho)  e^{-\rho^2} dV_{g}.
\end{equation}

Therefore, if $n \geq 3$ we get the stability estimate:
\begin{align*}
\mbox{dist}(M, \hn_c) \leq 4 \delta(M),
\end{align*}
since $y=\mbox{tr} \ \mathcal{D}\geq 0$ by Hessian comparison. Now we assume $n = 2.$  The first term in \eqref{30-7-8} is precisely $4 \delta(M).$ So, we only need to estimate the later term. Writing $k^2=-c, \kappa_0^2 = -c_0$ and using $M$ is pinched between $c_0$ and $c$ we get
$$0 \leq y(\rho) = \Delta \rho - (n-1) \kappa \coth(\kappa \rho) \leq (n-1) \left( \kappa_0 \coth(\kappa_0 \rho) - \kappa \coth(\kappa \rho) \right) = O(1), \ \ \mbox{if} \ \rho \geq \delta.$$
 Further, using Lemma~\ref{laplace-0}, near $\rho \approx 0,$ and $\text{{\bf ct}}_c(\rho) \approx 1/\rho,$ we get
 \begin{align*}
\Delta \rho = \frac{n-1}{\rho} - \frac{Ric_M(\partial_\rho, \partial_\rho)}{3}\rho + O(\rho^2) 
 \end{align*}
and 
\begin{align*}
(n-1)\kappa \coth(\kappa \rho) = (n-1)\kappa \left( \frac{1}{\kappa \rho} + \frac{\kappa \rho}{3} + O(\rho^3) \right) = \frac{n-1}{\rho} + \frac{(n-1)\kappa^2}{3}\rho + O(\rho^3)
\end{align*}
and therefore, $y(\rho) = O(\rho).$
Let $R > 0$ small and $\rho < R,$ then from the above estimates it follows  $y(\rho) \text{{\bf ct}}_c(\rho) = O(1)$ and the volume element $dV_{g} = J(\rho, \theta) d\rho d\theta \approx \rho d\rho d\theta$. So, near the pole:
\begin{align*}
\int_{B_R} y(\rho) \text{{\bf ct}}_c(\rho) dV_{g} \approx \int_0^{2\pi} \int_0^R C \rho \, d\rho d\theta = O(R^2).
\end{align*}
Now away from the pole, using $\frac{\coth t}{t}$ decreasing in $t,$ we get

\begin{align*}
\int_{M \setminus B_R} y(\rho) \text{{\bf ct}}_c(\rho)  e^{-\rho^2} dV_{g} &= 
\int_{M \setminus B_R} (\kappa\rho y(\rho)) \frac{\kappa\coth \kappa\rho}{\kappa\rho} e^{-\rho^2} dV_{g} \\
& \leq \frac{\kappa^2\coth \kappa R}{\kappa R} \int_M \rho y(\rho) e^{-\rho^2} dV_{g} = O\left(\frac{\delta(M)}{R^2}\right)
\end{align*}
where we used $\coth t = O(\frac{1}{t}).$ Optimizing $R^2 + \frac{\delta(M)}{R^2}$ by taking 
\begin{align*}
R^2 = \frac{\delta(M)}{R^2}, \ \ \mbox{gives} \ \ R^2 = \delta(M)^{\frac{1}{2}}.
\end{align*}

As a result, we get $\int_{M \setminus B_R} y(\rho) \text{{\bf ct}}_c(\rho)  e^{-\rho^2} dV_{g} = O(\delta(M)^{\frac{1}{2}}).$
Hence, from \eqref{30-7-8}, we get
\begin{align*}
\text{dist}(M, \hn_c) = \int_M E(\rho) e^{-\rho^2} dV_{g} \lesssim 4 \delta(M) + \delta(M)^{\frac{1}{2}} \lesssim \delta(M)^{\frac{1}{2}}.
\end{align*}
This completes the proof.
\end{proof}

\subsection{Optimality of the stability estimate for $n\geq 3$}

We construct the following examples:

\medskip

\medskip

Let $c=-\kappa^2$ and define $M := \hn_{c - \epsilon}, \,\epsilon >0.$ Then clearly, $\text{dist}(M, \hn_c) = C \epsilon.$ The Laplacian of the distance function on $M$ is $\Delta_g \rho = (n-1)\tilde{\kappa} \coth(\tilde{\kappa}\rho)$, where 

$$\tilde{\kappa} = \sqrt{\kappa^2 + \epsilon} = \kappa \left( 1 + \frac{\epsilon}{\kappa^2} \right)^{1/2}
 = \kappa + \frac{\epsilon}{2\kappa} + O(\epsilon^2).$$
The model term is $(n-1)\kappa \coth(\kappa\rho)$. Their difference is:
$$y(\rho) = (n-1) \Big[ \tilde{\kappa} \coth(\tilde{\kappa}\rho) - \kappa \coth(\kappa\rho) \Big].$$

Set $f(s) = s \coth(s\rho)$ and Taylor expand around $s = \kappa$. The derivative with respect to $s$ is: 
 
 $$\frac{d}{ds} \Big[ s \coth(s\rho) \Big] = \coth(s\rho) + s  \left( -\text{csch}^2(s\rho) \rho \right) = \coth(s\rho) - s\rho \text{csch}^2(s\rho).$$
 Therefore, 
 $$\tilde{\kappa} \coth(\tilde{\kappa}\rho) - \kappa \coth(\kappa\rho) = \left[ \coth(\kappa\rho) - \kappa\rho \text{csch}^2(\kappa\rho) \right] \left( \frac{\epsilon}{2\kappa} \right) + O(\epsilon^2),$$

and hence, 
 
 $$y(\rho) = \frac{\epsilon(n-1)}{2\kappa} \Big[ \coth(\kappa\rho) - \kappa\rho \text{csch}^2(\kappa\rho) \Big] + O(\epsilon^2).$$
 
The deficit $\delta(M)$ is then given by

\begin{align*}
\delta(M) &= \frac{1}{2} \int_M \rho \, y(\rho) \, e^{-\rho^2} dV_{g} \\
&= \frac{1}{2} \int_0^\infty \rho \left( \frac{\epsilon(n-1)}{2\kappa} \Big[ \coth(\kappa\rho) - \kappa\rho \text{csch}^2(\kappa\rho) \Big]  + O(\epsilon^2)\right) e^{-\rho^2} \omega_{n-1} \left( \frac{\sinh(\tilde\kappa\rho)}{\tilde\kappa} \right)^{n-1} d\rho, \\
&= \epsilon \left[ \frac{(n-1)\omega_{n-1}}{4\kappa} \int_0^\infty \rho \Big[ \coth(\kappa\rho) - \kappa\rho \text{csch}^2(\kappa\rho) \Big] e^{-\rho^2} \left( \frac{\sinh(\kappa\rho)}{\kappa} \right)^{n-1} d\rho \right] + O(\epsilon^2).
\end{align*}

We need to make sure that the integral is finite. Note that we have the following expansion: 
$\coth(s) = \frac{1}{s} + \frac{s}{3} + O(s^2)$ and $\text{csch}^2(s) = \frac{1}{s^2} - \frac{1}{3} + O(s^2)$. Therefore, for small $\rho$,

\begin{align*}
\coth(\kappa\rho) - \kappa\rho \text{csch}^2(\kappa\rho) &= \left( \frac{1}{\kappa\rho} + \frac{\kappa\rho}{3} \right) - \kappa\rho \left( \frac{1}{\kappa^2\rho^2} - \frac{1}{3} \right) + O(\rho^2) \\
&= \left( \frac{1}{\kappa\rho} + \frac{\kappa\rho}{3} \right) - \left( \frac{1}{\kappa\rho} - \frac{\kappa\rho}{3} \right) + O(\rho^2)= \frac{2\kappa\rho}{3} + O(\rho^2).
\end{align*}

Consequently, we have $\delta(M) \approx C \epsilon$.

\section{Manifolds with non-negative Ricci curvature}

Let $(M, g)$ be a complete, non-compact, $n$-dimensional Riemannian manifold with non-negative Ricci curvature ($\operatorname{Ric_M} \geq 0$). Recall the definition of the asymptotic volume ratio (AVR):
\begin{align*}
\theta := \operatorname{AVR}(M) := \lim_{r \to \infty} \frac{V_g(B(x_0, r))}{\omega_n r^n} \leq 1.
\end{align*}
By the Bishop–Gromov volume comparison theorem, this limit is well-defined and independent of the choice of base point $x_0 \in M$. Throughout this section, we assume that $\theta > 0$, meaning that $M$ exhibits Euclidean volume growth.

For such manifolds, Balogh and Krist\'aly \cite[Proposition 3.1]{BK} derived a P\'olya–Szeg\"o inequality. To state it, let $u \in H^1(M)$ be a non-negative function. We define its decreasing Schwarz rearrangement, $u^\star : \mathbb{R}^n \to [0, \infty)$, centered at the origin. The function $u^\star$ is constructed such that it is radial, radially non-increasing, and equimeasurable with $u$. That is, for all $t \geq 0$, the superlevel sets satisfy:
\begin{align*}
\operatorname{Vol}_{\mathbb{R}^n}(\{y \in \mathbb{R}^n \mid u^\star(y) > t\}) = \operatorname{Vol}_g(\{x \in M \mid u(x) > t\}).
\end{align*}
Denote the superlevel sets by $E_t := \{x \in M \mid u(x) > t\}$ and volume $V(t) := \operatorname{Vol}_g(E_t)$, and $r_t$ be the radius of the Euclidean ball with the same volume: $r_t := \left( \frac{V(t)}{\omega_n} \right)^{1/n}$ and define the radial function $u^\star$ by the relation $u^\star(r_t) = t$ for all $t$.
Then the following $\theta$-weighted Pólya–Szeg\"o inequality holds:
\begin{align*}
\int_M \vert{}\nabla_g u\vert{}^2 \, dV_{g} \ge \theta^{2/n} \int_{\mathbb{R}^n} \vert{}\nabla u^\star\vert{}^2 \, dy.
\end{align*}
Furthermore, equality holds (for some smooth $u$) if and only if $M$ is isometric to the Euclidean space $\mathbb{R}^n$. Using the layer cake representation, we also obtain the inequality
\begin{align} \label{moment comparison}
  \int_M d_{x_0}(x)^2 u(x)^2 \, dV_{g} \ge \int_{\mathbb{R}^n} \vert{}y\vert{}^2 u^\star(y)^2 \, dy.  
\end{align}

 Indeed, using layer-cake representation we can write
\begin{align*}
\int_M d_{x_0}(x)^2 u(x)^2 \, dV_{g} = 2 \int_0^\infty t \left( \int_{E_t} d_{x_0}(x)^2 \, dV_{g} \right) dt.
\end{align*}

By the bathtub principle, it is easy to see that, for each $t > 0$, the inner integral is minimized by a geodesic ball $B_{x_0}(R_t)$ satisfying 
$\operatorname{Vol}_g(B_{x_0}(R_t)) = V(t)$, and the Bishop-Gromov volume comparison theorem guaranties $R_t \ge r_t = \left( \frac{V(t)}{\omega_n} \right)^{1/n}$. Consequently, we obtain the lower bound:

\begin{align*}
\int_{E_t} d_{x_0}(x)^2 \, dV_{g} \ge \int_{B_{x_0}(R_t)} d_{x_0}(x)^2 \, dV_{g} \ge \int_{B_{0}(r_t)} \vert{}y\vert{}^2 \, dy,
\end{align*}
where $B_0(r_t) = \{y \in \mathbb{R}^n \mid u^\star(y) > t\} \subset \mathbb{R}^n$ is the Euclidean ball of radius $r_t$. Integrating this final lower bound with respect to $2t \, dt$ recovers the right-hand side of the desired inequality on $\mathbb{R}^n$.

\subsection{Non-quantitative stability}

{\bf Proof of Theorem~\ref{th:stability-3}:}
\begin{proof}
\noindent
{\bf Step 1:}  Let $u_k$ be an almost extremizer, and let $u_k^\star : \mathbb{R}^n \to \mathbb{R}_{\ge 0}$ denote its Schwarz symmetrization. Then it preserves the $L^2$-norm 
$$\Vert{}u_k^\star\Vert{}_{L^2(\mathbb{R}^n)} = \Vert{}u_k\Vert{}_{L^2(M_k)} = 1.$$

Since $\operatorname{Ric}_{M_k} \ge 0$,
 \begin{align*}
 E_{M_k}(u_k) := \int_{M_k} \vert{}\nabla u_k\vert{}^2 dV_{g_k} \ge \operatorname{AVR}(M_k)^{2/n} \int_{\mathbb{R}^n} \vert{}\nabla u_k^\star\vert{}^2 dx =: \operatorname{AVR}(M_k)^{2/n} E_{\mathbb{R}^n}(u_k^\star),
 \end{align*}
 
 and by \eqref{moment comparison}
 
 $$V_{M_k}(u_k) := \int_{M_k} d_{M_k}(o_k, x)^2 u_k(x)^2 dV_{g_k} \ge \int_{\mathbb{R}^n} \vert{}x\vert{}^2 (u_k^\star(x))^2 dx = V_{\mathbb{R}^n}(u_k^\star).$$
 
 We rewrite $1 + \delta(u_k; M_k,o_k)$ as follows:
 \begin{align*}
     1 + \delta(u_k; M_k,o_k) &= \frac{4}{n^2 \operatorname{AVR}(M_k)^{2/n}} E_{M_k}(u_k) V_{M_k}(u_k)\\ &= \underbrace{\left( \frac{E_{M_k}(u_k)}{\operatorname{AVR}(M_k)^{2/n} E_{\mathbb{R}^n}(u_k^\star)} \right)}_{A_k \ge 1} \cdot \underbrace{\left( \frac{V_{M_k}(u_k)}{V_{\mathbb{R}^n}(u_k^\star)} \right)}_{B_k \ge 1} \cdot \underbrace{\left( \frac{4}{n^2} E_{\mathbb{R}^n}(u_k^\star) V_{\mathbb{R}^n}(u_k^\star) \right)}_{C_k \ge 1}.
 \end{align*}

 Since $1 + \delta(u_k; M_k,o_k) \longrightarrow 1$ and $A_k \ge 1,\, B_k \ge 1,\, C_k \ge 1$, up to a subsequence, we have
 \begin{align} \label{limits}
     \lim_{k \to \infty} \frac{E_{M_k}(u_k)}{\operatorname{AVR}(M_k)^{2/n} E_{\mathbb{R}^n}(u_k^\star)} = 1, \ \ \lim_{k \to \infty} \frac{V_{M_k}(u_k)}{V_{\mathbb{R}^n}(u_k^\star)} = 1, \ \ \lim_{k \to \infty} C_k=1.
 \end{align}

 Moreover, since $C_k = 1 + \delta(u_k^\star; \mathbb{R}^n,0) $, we obtain  $\delta(u_k^\star; \mathbb{R}^n,0) \longrightarrow 0$, i.e. $u_k^\star$ is almost an Euclidean HPW optimizer.

 \medskip

 \noindent
 {\bf Step 2:} By \cite[Corollary 2.3]{MV}, there exists $\lambda_k > 0$ such that the rescaled functions$$v_k(y) := \lambda_k^{-n/2} u_k^\star\left(\frac{y}{\lambda_k}\right)$$satisfy $\Vert{}v_k\Vert{}_{L^2(\mathbb{R}^n)} = 1$ and converge strongly in $H^1(\mathbb{R}^n) \cap L^2(\mathbb{R}^n, \vert{}y\vert{}^2 dy)$ to an optimizer of the Euclidean HPW inequality; without loss of generality, we may take the limit to be the standard Gaussian $G(y) = \pi^{-n/4} e^{-\vert{}y\vert{}^2/2}.$ The $\lambda_k$ are chosen such that $\int_{\R^n}|y|^2|v_k|^2dy=1$. By change of variable $\int_{\R^n}|y|^2|v_k|^2dy=\lambda_k^2 V_{\R^n}(u_k^\star)$ and using \eqref{limits} we conclude
 \begin{align*}
    1 = \lim_{k \to \infty} \frac{V_{M_k}(u_k)}{V_{\mathbb{R}^n}(u_k^\star)} = \lim_{k\to\infty}\lambda_k^2V_{M_k}(u_k).
 \end{align*}

  \medskip

 \noindent
{\bf Step 3:} Define the rescaled Riemannian manifold $\tilde{M}_k := (M_k, \tilde{g}_k)$ with $\tilde{g}_k := \lambda_k^2 g_k.$ Then 
 \begin{align} \label{rescaled relation}
 d_{\tilde M_k} = \lambda_k d_{M_k}, \ \ \ \ \ \ \ dV_{\tilde g_k} = \lambda_k^n \ dV_{g_k} \ ,
 \end{align}
and define the corresponding rescaled functions $\tilde{u}_k := \lambda_k^{-n/2} u_k$. Then $v_k = \tilde{u}_k^\star$ is the Schwarz symmetrization of $\tilde{u}_k$. Indeed, note that 
 \begin{align*}
 \operatorname{Vol}_{\tilde g_k}(\{\tilde u_k > t\}) = \lambda_k^n   \operatorname{Vol}_{ g_k}(\{ u_k > t \lambda_k^{\frac{n}{2}}\})
 =  \lambda_k^n  |\{ u_k^\star > t \lambda_k^{\frac{n}{2}}\}|.
 \end{align*}
On the other hand, by change of variable formula in $\mathbb{R}^n$
\begin{align*}
|\{ v_k > t \}| = |\{ u_k^\star(\lambda_k^{-1} \cdot) > t \lambda_k^{\frac{n}{2}}\}| = \lambda_k^n|\{ u_k^\star > t \lambda_k^{\frac{n}{2}}\}| . 
\end{align*}
Therefore, the super-level set of $\tilde{u}_k$ at level $t$ has volume $\lambda_k^n\operatorname{Vol}_{ g_k}(\{ u_k> t \lambda_k^{n/2}\}).$ Hence, the corresponding Euclidean radii satisfy $\tilde{r}_k(t) = \lambda_k r_k(t \lambda_k^{n/2})$. Applying the level set relation $u_k^\star(r_k(s)) = s$ for $s = t \lambda_k^{n/2}$ yields 
\begin{align*}
v_k(\tilde{r}_k(t)) = \lambda_k^{-n/2} u_k^\star\left(\frac{\tilde{r}_k}{\lambda_k}\right) = \lambda_k^{-n/2} u_k^\star(r_k(t \lambda_k^{n/2})) = \lambda_k^{-n/2}(t \lambda_k^{n/2}) = t,
\end{align*}
 proving $v_k = \tilde{u}_k^\star$. It is also easy to verify using \eqref{rescaled relation} that the following holds:
\begin{align*}
 V_{\tilde{M}_k}(\tilde{u}_k) = \lambda_k^2 V_{M_k}(u_k), \quad V_{\mathbb{R}^n}(v_k) = \lambda_k^2 V_{\mathbb{R}^n}(u_k^\star).
 \end{align*}
 
 Since $B_k \to 1$ by \eqref{limits} and $v_k$ converges strongly to the Gaussian $G$, it follows that 
 \begin{align*}
  V_{\tilde{M}_k}(\tilde{u}_k) - V_{\mathbb{R}^n}(v_k) = o(1), \ \ \mbox{as} \ k \rightarrow \infty.
 \end{align*}

  \medskip

{\bf Step 4}:
To simplify notations, we denote:
 \begin{align*}
  \mathcal{V}_k(s) :=  \operatorname{Vol}_{\tilde{g}_k}(B_{\tilde{M}_k}(o_k, s)),
 \end{align*}
 the volume of ball of radius $s$ centered at $o_k$ in $\tilde{M}_k$.
By Cavalieri's principle, denoting the super-level sets $E_{tk} := \{x \in \tilde{M}_k : \tilde{u}_k(x) > t\}$, the second moment is given by
\begin{align*}
V_{\tilde{M}_k}(\tilde{u}_k) = \int_0^\infty \left( \int_{E_{tk}} d_{\tilde{M}_k}(o_k, x)^2 dV_{\tilde{g}_k}(x) \right) d(t^2).
\end{align*}
By bathtub principle, for each $t > 0$, the inner integral is minimized by a geodesic ball $B_{\tilde{M}_k}(o_k, R_{tk})$ satisfying $ \operatorname{Vol}_{\tilde{g}_k}(B_{\tilde{M}_k}(o_k, R_{tk})) =  \operatorname{Vol}_{\tilde{g}_k} (E_{tk}) = \mathcal{V}_k(R_{tk})$. By the Bishop–Gromov comparison theorem, $R_{tk} \ge r_{tk} := \left(\frac{\mathcal{V}_k(R_{tk})}{\omega_n}\right)^{1/n}$. Recalling the notation $\mathcal{V}_k(s) :=  \operatorname{Vol}_{\tilde{g}_k}(B_{\tilde{M}_k}(o_k, s))$, and using the co-area formula (applied to the distance function) we rewrite 
\begin{align*}
\int_{B_{\tilde{M}_k}(o_k, R_{tk})} d_{\tilde{M}_k}(o_k, x)^2 dV_{\tilde{g}_k} = \int_0^{R_{tk}} s^2 {d\mathcal{V}_k(s)}.
\end{align*}

We do the same for the Euclidean space and denote their relative difference by
\begin{align*}
\Delta_k(t) := \int_0^{R_{tk}} s^2 d\mathcal{V}_k(s) - \int_0^{r_{tk}} s^2 d(\omega_n s^n).
\end{align*}
Applying integration by parts, 
\begin{align*}
\int_0^{R_{tk}} s^2 d\mathcal{V}_k(s) = \Big[ s^2 \mathcal{V}_k(s) \Big]_0^{R_{tk}} - \int_0^{R_{tk}} 2s \mathcal{V}_k(s) \, ds =  R_{tk}^2 \mathcal{V}_k(R_{tk}) - \int_0^{R_{tk}} 2s \mathcal{V}_k(s) \, ds.
\end{align*}

Similarly, for the Euclidean counterpart, using $\omega_n r_{tk}^n = \mathcal{V}_k(R_{tk}),$
\begin{align*}
\int_0^{r_{tk}} s^2 d(\omega_n s^n) = \Big[ s^2 (\omega_n s^n) \Big]_0^{r_{tk}} - \int_0^{r_{tk}} 2s (\omega_n s^n) \, ds =  r_{tk}^2 \mathcal{V}_k(R_{tk}) - \int_0^{r_{tk}} 2s (\omega_n s^n) \, ds,
\end{align*}
Grouping the boundary terms together, and since $R_{tk} \ge r_{tk}$ (by Bishop–Gromov), we obtain 
\begin{align} \label{est1}
\Delta_k(t) &= \mathcal{V}_k(R_{tk})(R_{tk}^2 - r_{tk}^2) - \int_0^{R_{tk}} 2s \mathcal{V}_k(s) \, ds + \int_0^{r_{tk}} 2s (\omega_n s^n) \, ds \notag \\
 &= \mathcal{V}_k(R_{tk})(R_{tk}^2 - r_{tk}^2) - \int_{r_{tk}}^{R_{tk}} 2s \mathcal{V}_k(s) \, ds + \int_0^{r_{tk}} 2s \left( \omega_n s^n - \mathcal{V}_k(s) \right) ds.
\end{align}
Since $\mathcal{V}_k(s) =  \operatorname{Vol}_{\tilde{g}_k}(B_{\tilde{M}_k}(o_k, s))$ is a non-decreasing function of $s$, for any radius $s \in [r_{tk}, R_{tk}]$ we have $\mathcal{V}_k(s) \le \mathcal{V}_k(R_{tk}) $. Therefore, 
\begin{align}\label{est2}
\int_{r_{tk}}^{R_{tk}} 2s \mathcal{V}_k(s) \, ds \le \int_{r_{tk}}^{R_{tk}} 2s \mathcal{V}_k(R_{tk}) \, ds = \mathcal{V}_k(R_{tk}) \int_{r_{tk}}^{R_{tk}} 2s \, ds  = \mathcal{V}_k(R_{tk})(R_{tk}^2 - r_{tk}^2).
\end{align}
By \eqref{est1} and \eqref{est2} we conclude
\begin{align*}
\Delta_k(t) \ge \int_0^{r_{tk}} 2s \left( \omega_n s^n - \mathcal{V}_k(s) \right) ds \ge 0.
\end{align*}
As a result, we deduce
\begin{align} \label{est3}
V_{\tilde{M}_k}(\tilde{u}_k) - V_{\mathbb{R}^n}(v_k) = \int_0^\infty \Delta_k(t) d(t^2) \ge \int_0^\infty \left( \int_0^{r_{tk}} 2s \left( \omega_n s^n - \mathcal{V}_k(s) \right) ds \right) d(t^2) \ge 0,
\end{align}
where recall by Bishop-Gromov $\mathcal{V}_k(s) =  \operatorname{Vol}_{\tilde{g}_k}(B_{\tilde{M}_k}(o_k, s)) \leq \omega_ns^n.$

\medskip
\noindent
{\bf Step 5:}  Since the left hand side of \eqref{est3} goes to $0$ as $k \rightarrow \infty,$ passing to a subsequence, for $L^1$-almost every $t > 0$:
\begin{align*}
\int_0^{r_{tk}} 2s \left( \omega_n s^n - \mathcal{V}_k(s) \right) ds \longrightarrow 0.
\end{align*}
Since the integrand $2s(\omega_n s^n - \mathcal{V}_k(s))$ is non-negative everywhere, this $L^1$-convergence imply the pointwise limit $\mathcal{V}_k(s) \longrightarrow \omega_n s^n$ for almost every $s \in (0, r_{tk})$. Since the limit of $v_k$ is Gaussian, it is supported everywhere, and hence $r_{tk} \to \infty$ as $t \to 0$. This implies the pointwise limit holds for almost every $s \in (0, \infty)$. Therefore, the volume density 
\begin{align*}
\tilde{\theta}_k(s) := \frac{\mathcal{V}_k(s)}{\omega_n s^n} \rightarrow 1, \ \ \mbox{as} \ k \rightarrow \infty,
\end{align*}
 almost everywhere. As the density $\tilde{\theta}_k(R)$ is monotonically non-increasing in $R$, for any fixed $R_0 > 0$, choosing $s_0 > R_0$ where $\tilde{\theta}_k(s_0) \to 1$ holds, we get 
 \begin{align*}
 1 \ge \tilde{\theta}_k(R_0) \ge \tilde{\theta}_k(s_0) \longrightarrow 1, \ \ \mbox{as} \ k \rightarrow \infty.
 \end{align*}
Hence $\lim_{k \to \infty} \tilde{\theta}_k(R_0) = 1$ for every $R_0 > 0$ and since the limit is continuous,  by Dini's theorem, $\tilde{\theta}_k(R) \longrightarrow 1$ uniformly on compact subsets of $(0, \infty)$.

\noindent
{\bf Step 6:} Applying the  compactness theorem of Colding \cite[Theorem 0.8]{Col} we conclude that $\tilde{M}_k$ converges  in pointed Gromov-Hausdorff topology to the Euclidean space. This completes the proof. 
\end{proof}

\subsection{Quantitative Rigidity}

{\bf Proof of Theorem~\ref{th:stability-2}:}
\begin{proof}
For notational simplicity we denote $\mbox{AVR}(M)$ by $\theta$ and let $u_\alpha(x)=e^{-\frac{\alpha}{2}d_{x_0}(x)^2},$ for $\alpha \geq 1$ and $\alpha$ to be chosen later. We divide the proof into several steps. In the first two steps, we show that for every $M,$ there exists a $\alpha_M$ large, such that the desired estimate holds. Then we use a compactness argument to conclude our proof. 

\vspace{2mm}

{\bf Step 1:}  In this step we  show that  $\big(\frac{\pi}{\alpha}\big)^{n/2} \theta \leq \|u_\alpha\|^2_{L^2(M)}\leq (\frac{\pi}{\alpha})^\frac{n}{2}$.

To see this, if  we denote by $u_\alpha^{\star}$ the spherical decreasing rearrangement of $u_\alpha$ then, the superlevel sets $\{x\in M:  u_\alpha > t\}, t \leq 1$ are geodesic balls $B_{x_0}(R_\alpha(t))$ with $R_{\alpha}(t) = \sqrt{-\frac{2}{\alpha} \ln t}$. Let $V_\alpha(t) = \text{Vol}_g(\{ u_\alpha > t\}).$ By Bishop-Gromov monotonicity, 
\begin{align*}
\theta \omega_n R_\alpha(t)^n \le V_\alpha(t) \le \omega_n R_\alpha(t)^n.
\end{align*}

Let $\{y\in\R^n: u_\alpha^{\star} > t\}$ be the Euclidean ball of radius $r_\alpha(t),$ then by definition of the rearranged set, it satisfies $r_\alpha(t) = \left( \frac{V_\alpha(t)}{\omega_n} \right)^{1/n}$, and hence
\begin{align} \label{compare R and r}
\theta^{1/n} R_\alpha(t) \leq r_\alpha(t) \leq R_\alpha(t).
\end{align}

Since $u_\alpha^\star(r_\alpha(t)) = t = e^{-\frac{\alpha}{2} R_\alpha(t)^2}$, by \eqref{compare R and r}, we have the following pointwise bound on $u_\alpha^\star(r(t))$ and hence for every $r = |y|, y \in \mathbb{R}^n$:
\begin{equation}\label{e-comp}
e^{-\frac{\alpha}{2}\theta^{-2/n} r^2} \leq u_\alpha^\star(r) \leq e^{-\frac{\alpha}{2}  r^2}.
\end{equation}
Using \eqref{e-comp} we see that
$$\int_{M}|u_\alpha|^2\;dV_{g}=\int_{\R^n}|u^\star_\alpha|^2dy\leq \int_{\R^n}e^{-\alpha  r^2}dy=\bigg(\frac{\pi}{\alpha}\bigg)^\frac{n}{2}.$$
Similarly the opposite inequality holds.
\medskip

{\bf Step 2:} In this step we will prove that for $\alpha = \alpha_M$ sufficiently large,
$$\mathrm{I}_M(u_\alpha)\geq \frac{n^2}{4}\bigg(\frac{\pi}{\alpha}\bigg)^n\theta^{2+1/n}(1-\theta^\frac{1}{n}).$$

First we observe that  $|\nabla_g d_{x_0}|=1$ implies $|\nabla_g u_\alpha|^2=\alpha^2d_{x_0}^2e^{-\alpha d_{x_0}^2}$. Therefore, 
\begin{align}\label{A-B}
\mathrm{I}_M(u_\alpha)&=\alpha^2\bigg(\displaystyle\int_M d_{x_0}^2e^{-\alpha d_{x_0}^2}dV_{g} \bigg)^2-\frac{n^2}{4}\theta^{2/n}\bigg(\displaystyle\int_{M}e^{-\alpha d_{x_0}^2}dV_{g}\bigg)^2\nonumber\\
&=\bigg(A + \frac{n}{2} \theta^{1/n} B\bigg)\bigg(A - \frac{n}{2} \theta^{1/n} B\bigg),
\end{align}
where $$A :=\alpha \int_M d_{x_0}^2e^{-\alpha d_{x_0}^2}dV_{g} \quad\mbox{and}\quad B:=\int_{M}e^{-\alpha d_{x_0}^2}dV_{g}.$$
Now set,
$$v(r):=\frac{\mbox{Vol}_g(B_{x_0}(r))}{\omega_nr^n}, \quad S(r):=\frac{\mbox{Area}_g(\partial B_{x_0}(r))}{n\omega_nr^{n-1}}.$$
Then it also follows that
$$  S(0)=1,\, 0\leq S(r)\leq 1,  \, S \, \mbox{is monotonically decreasing in}\,\, r,\, \lim_{r\to\infty}S(r)=\theta.$$
Using co-area formula, $A$ and $B$ can be simplified as below
$$A=\alpha\int_0^\infty r^2e^{-\alpha r^2}\mbox{Vol}_g(\partial B_{x_0}(r))dr=n\omega_n\alpha\int_0^\infty r^{n+1}e^{-\alpha r^2}S(r)dr.$$
$$B=n\omega_n\int_0^\infty r^{n-1}e^{-\alpha r^2}S(r)dr.$$
Therefore, using \eqref{e-comp}
\begin{align}\label{5-8-4}
A + \frac{n}{2} \theta^{1/n} B\geq  \frac{n}{2} \theta^{1/n} B= \frac{n}{2} \theta^{1/n} \int_{M}|u_\alpha|^2dV_{g}
&=\frac{n}{2}\bigg(\frac{\pi}{\alpha}\bigg)^{n/2} \theta^{1+1/n}.
\end{align}
Using the definition of $A$ and $B$ and defining $r_0^2:=\frac{n}{2} \theta^{1/n}$, yields 
$$A - \frac{n}{2} \theta^{1/n} B=n\omega_n\int_0^\infty (\alpha r^2-r_0^2) r^{n-1}e^{-\alpha r^2}S(r)dr.$$
Now substituting $S(r)=\theta+(S(r)-\theta)$ in the above expression yields
 \begin{equation}\label{5-8-6}
A - \frac{n}{2} \theta^{1/n} B=\mathfrak{I}_1+\mathfrak{I}_2,
\end{equation}
where $$\mathfrak{I}_1:=n\omega_n\theta\int_0^\infty (\alpha r^2-r_0^2) r^{n-1}e^{-\alpha r^2}dr,$$
and $$\mathfrak{I}_2:=n\omega_n\int_0^\infty (\alpha r^2-r_0^2) r^{n-1}e^{-\alpha r^2}(S(r)-\theta)dr.$$
Substituting the value of $r_0^2$ and using the following two straight-forward integrals
\begin{equation}\label{5-8-9}
n\omega_n\int_0^\infty r^{n-1}e^{-\alpha r^2}dr=\left(\frac{\pi}{\alpha}\right)^{n/2} \quad\mbox{and}\quad n\omega_n\int_0^\infty r^{n+1}e^{-\alpha r^2}dr=\frac{n}{2} \left(\frac{\pi}{\alpha}\right)^{n/2}
\end{equation}
we obtain
 \begin{equation}\label{5-8-7}
\mathfrak{I}_1=\frac{n}{2}\bigg(\frac{\pi}{\alpha}\bigg)^{n/2} \theta(1-\theta^{1/n}).
\end{equation}
Now we argue that for sufficiently large $\alpha,$ one has $\mathfrak{I}_2 >0.$
Note that as $S$ is a nonnegative, non-increasing function taking values in $(\theta, 1)$, it's of bounded variation, and thus uniquely generates a non-negative Radon measure $\mu=-dS$ on $(0,\infty)$. Then, by the Lebesgue-Stieltjes representation theorem $$S(r) - \theta = \int_{(r, \infty)} d\mu(s).$$ Now define, $f(r):=(\alpha r^2-r_0^2) r^{n-1}e^{-\alpha r^2}$. Since, $f\in L^1(0,\infty)$ and $\mu((0, \infty)) = S(0^+) - \theta \leq 1$, the product $F(r, s) = f(r) \cdot \chi_{\{0 < r < s\}}\in L^1(dr \otimes \mu)$. By Fubini's theorem
$\mathfrak{I}_2$ can be rewritten as
\begin{align}\label{5-8-8}
\mathfrak{I}_2=n\omega_n\int_{r=0}^\infty f(r)\bigg( \int_{(r, \infty)} d\mu(s)\bigg)dr&=n\omega_n\int_{s\in(0,\infty)}\bigg(\int_{r=0}^s f(r)dr\bigg)d\mu(s)\nonumber\\
&=n\omega_n\int_{s\in(0,\infty)} J(s)d\mu(s),
\end{align}
where, $$\quad J(s):=\int_0^s f(r)dr.$$
We observe that $J$ has the following properties: $J(0)=0$, $J$ is decreasing for $s\in (0,\frac{r_0}{\sqrt{\alpha}})$, $J$ is increasing for $s\in (\frac{r_0}{\sqrt{\alpha}},\infty)$ and using \eqref{5-8-8} a straight-forward computation yields
$$\lim_{s\to\infty}J(s)=\int_0^{\infty}f(r)dr =\frac{n}{4}\alpha^{-n/2}\Gamma\bigg(\frac{n}{4}\bigg)\big(\alpha-\theta^{1/n}\big)>0,$$
since $\alpha\geq 1$ and $\theta<1$ if $M$ is not isometric to $\R^n$. Therefore, by intermediate value theorem there exists $r^*_\alpha>\frac{r_0}{\sqrt{\alpha}}$ such that $J(r^*_\alpha)=0$. In fact, using the definition of gamma function $\gamma(a, x) := \int_0^x t^{a-1} e^{-t} dt$, we can simplify
$$J(s) = \frac{1}{2 \alpha^{n/2}} \left[ \gamma\left(\frac{n}{2} + 1, \alpha s^2\right) - r_0^2 \gamma\left(\frac{n}{2}, \alpha s^2\right) \right].$$
Substituting $r_0^2 = \frac{n}{2} \theta^{1/n}$ in the above expression gives: if  $J(s)=0$ then $s$ must satisfy: $$\frac{\gamma\left(\frac{n}{2} + 1, \alpha s^2\right)}{\gamma\left(\frac{n}{2}, \alpha s^2\right)} = \frac{n}{2} \theta^{1/n}.$$
Now using the gamma recurrence relation $\gamma(a+1, x) = a\gamma(a, x) - x^a e^{-x}$ with $a = \frac{n}{2}$ and $x = \alpha s^2$, we can rewrite the zero condition of $J$ as
$$\frac{(\alpha s^2)^{n/2} e^{-\alpha s^2}}{\gamma\left(\frac{n}{2}, \alpha s^2\right)} = \frac{n}{2} \left(1 - \theta^{1/n}\right).$$
By change of variables $x = \alpha s^2$, the root $s^* = \sqrt{x^*/\alpha}$ is uniquely determined by $x^*$, which satisfies 
\begin{equation}\label{5-8-10}
\frac{(x^*)^{n/2} e^{-x^*}}{ \gamma\left(\frac{n}{2}, x^*\right) } = \frac{n}{2}\left(1 - \theta^{1/n}\right).\end{equation}
Hence unique positive zero of the function $J$ is given by $r^*_\alpha:=\sqrt{x^*/\alpha}$, where $x^*$ is determined by \eqref{5-8-10}. The same gives $J(s)
=\frac{1}{2\alpha^{n/2}}
\left[
\frac{n}{2}\left(1-\theta^{1/n}\right)
\gamma\left(\frac{n}{2},x\right)
-x^{n/2}e^{-x}
\right],$ where  $x=\alpha s^2,$ and hence $J$ is uniformly bounded in $(0,\infty).$ 
Now using this  $r^*_\alpha$ we split the integration in \eqref{5-8-8} into two parts
$$\mathfrak{I}_2=n\omega_n\int_{s\in(0,r^*_\alpha)} J(s)d\mu(s)+n\omega_n\int_{s\in(r^*_\alpha, \infty)} J(s)d\mu(s).$$
Note that the second integral is strictly positive and the 1st integral is strictly negative. Since $r^*_\alpha\to 0$ as $\alpha\to\infty$, we can choose $\alpha>>1$ large enough such that 
\begin{equation}\label{5-8-11}
\mathfrak{I}_2=n\omega_n\int_{s\in(0,r^*_\alpha)} J(s)d\mu(s)+n\omega_n\int_{s\in(r^*_\alpha, \infty)} J(s)d\mu(s)>0.
\end{equation}
Hence substituting \eqref{5-8-11} and \eqref{5-8-7} into \eqref{5-8-6} yields
$$A - \frac{n}{2} \theta^{1/n} B\geq \frac{n}{2}\bigg(\frac{\pi}{\alpha}\bigg)^{n/2} \theta(1-\theta^{1/n}).$$
Combining this with \eqref{5-8-4} and substituting back into \eqref{A-B} gives
$$\mathrm{I}_M(u_\alpha)\geq \frac{n}{2}\bigg(\frac{\pi}{\alpha}\bigg)^{n/2} \theta^{1+1/n} \cdot  \frac{n}{2}\bigg(\frac{\pi}{\alpha}\bigg)^{n/2} \theta(1-\theta^{1/n})=\frac{n^2}{4}\bigg(\frac{\pi}{\alpha}\bigg)^n\theta^{2+1/n}(1-\theta^{1/n})$$

\noindent
{\bf Step 3}: For simplicity of notations, we denote
\begin{align*}
    \tilde{I}_M(u) = \delta \left(\frac{u}{\|u\|_{L^2}}; M, x_0\right),
\end{align*}
where $\delta$ is defined by \eqref{deficit for Ric}. Then 
 Step 1 and Step 2  yields
$$ \tilde I_M(u_\alpha)\geq \theta^{2-1/n}(1-\theta^{1/n}) \quad\mbox{for}\, \alpha\gg1.$$
Therefore, given any complete noncompact Riemannian manifold $(M,g)$ with $Ric\geq 0$, we can choose $\alpha=\alpha_M>1$ such that 
$$ \tilde I_M(u_{\alpha_M})\geq  \theta^{2-1/n}(1-\theta^{1/n}).$$

\noindent
{\bf Step 4:} 
 We show that $\sup_{\alpha\geq 1}\tilde I_M(u_\alpha) \leq \frac{4}{n^2}\theta^{-2/n}\left( 2 + \frac{2^{n/2}}{\theta} \right)^2.$

Using the crude bound $2t \leq e^t + 2,$ in the definition of $\tilde I(u_\alpha)$, we have
we 
\begin{align*}
   \dfrac{n^2}{4} \theta^{2/n}\left( \tilde I( u_\alpha) + 1\right) &= \frac{\left( \int_M \vert{}\nabla u_\alpha\vert{}^2 dV_g \right) \left( \int_M d_{x_0}^2 u_\alpha^2 dV_g \right)}{\left( \int_M u_\alpha^2 dV_g \right)^2} \\
   &= \left( \frac{\alpha \int_M d_{x_0}^2 e^{-\alpha d_{x_0}^2} dV_g}{\int_M e^{-\alpha d_{x_0}^2} dV_g} \right)^2 \\
   &\leq \left( \frac{\int_M  (e^{-\frac{\alpha}{2} d_{x_0}^2} + 2e^{-\alpha d_{x_0}^2})dV_g}{\int_M e^{-\alpha d_{x_0}^2} dV_g} \right)^2\\
   &\leq \left( 2 + \frac{2^{n/2}}{\theta} \right)^2,
\end{align*}
where we used Step 1 to evaluate the last integrals.

\medskip

{\bf Step 5:} We claim that for $\theta_{*}>0$ there exists a constant $C = C(n,\theta_{*})$ independent of the manifold such that  
\begin{align*}
\inf_{\alpha\geq 1}\tilde I_M(u_\alpha) \geq \; C\sup_{\alpha\geq 1}\tilde I_M(u_\alpha),
\end{align*}
holds for all $M$ satisfying $\mbox{AVR}(M) \geq\theta_{*}.$ To emphasize the dependence on the metric, we denote the AVR by $\mbox{AVR}(M,g).$  Consider the rescaled metric $g_\alpha := \alpha g$ on $M$. Then
\begin{align*}
d_{g_\alpha}(o, x) = \sqrt{\alpha} d_g(o, x), \  dV_{g_\alpha} = \alpha^{n/2} dV_g, \  \mbox{AVR}(M, g_\alpha) = \mbox{AVR}(M, g).    
\end{align*}

It is easy to verify that 
\begin{align*}
    \tilde I(u_1)_{M, g_\alpha} = \frac{\alpha^{n+2}}{\alpha^n} \frac{4 \left( \int_M d_g^2 e^{-\alpha d_g^2} dV_g \right)^2}{\frac{n^2}{4} \mbox{AVR}(M)^{2/n}\left( \int_M e^{-\alpha d_g^2} dV_g \right)^2} = \tilde I(u_\alpha)_{M, g}.
\end{align*}

Define the space of pointed Riemannian manifolds:
\begin{align*}
    \mathcal{M}(n, \theta_{*}) := \left\{ (N, h, o) : N \text{ is complete, Riemannian $n$-manifold},  \ Ric_h \ge 0, \ \mbox{AVR}(N, h) \ge \theta_{*} \right\}.
\end{align*}
Because $ Ric_{g_\alpha} = \frac{1}{\alpha} Ric_M \ge 0$ and $\mbox{AVR}(M, g_\alpha) = \mbox{AVR}(M, g) \ge \theta_{*}$, the rescaled manifold $(M, g_\alpha, o) \in \mathcal{M}(n, \theta_{*})$ for all $\alpha \ge 1$.Therefore, for a fixed manifold $M$:
\begin{align*}
\sup_{\alpha \ge 1} \tilde I(u_\alpha)_{M, g} &= \sup_{\alpha \ge 1} \tilde I( u_1)_{M, g_\alpha} \le \sup_{(N, h, o_N) \in \mathcal{M}(n, \theta_{*})}\tilde I( u_1)_{N, h} =: M_0(n, \theta_{*}), \\[8pt]
\inf_{\alpha \ge 1} \tilde I(u_\alpha)_{M, g} &= \inf_{\alpha \ge 1}\tilde I( u_1)_{M, g_\alpha} \ge \inf_{(N, h, o_N) \in \mathcal{M}(n, \theta_{*})} \tilde I( u_1)_{N, h} =: m_0(n, \theta_{*}).
\end{align*}

By step 4, we have $M_0(n, \theta_{*}) <\infty$ and by step 3, $M_0(n, \theta_{*}) >0$.  Now we claim  $m_0(n, \theta_{*}) >0$.

Suppose that the claim does not hold.  Then $m_0(n, \theta_{*}) = 0$ i.e., there exists a sequence $(N_k, h_k, o_k) \in \mathcal{M}(n, \theta_{*})$ such that $\tilde I( u_1)_{N_k, h_k} \to 0$.  Under $Ric_{h} \ge 0$, the Bishop–Gromov volume monotonicity \cite[Lemma 7.1.4]{Peter} ensures that for all balls $B_r(x) \subset B_R(x)$
$$\frac{\mathrm{Vol}_h(B_R(x))}{\mathrm{Vol}_h(B_r(x))} \le \left(\frac{R}{r}\right)^n,$$ 
This gives uniform doubling property on balls across all manifolds in $\mathcal{M}(n, \theta_{*})$. Furthermore, as $\operatorname{AVR}(N, h) \ge \theta_{*}$ it also holds $ \theta_{*}\omega_n\leq \mathrm{Vol}_h(B_x(1))\leq \omega_n$ for all $(N,h,o)\in \mathcal{M}(n, \theta_{*})$. Therefore, 
by Gromov compactness theorem (see \cite[Theorem 27.32(ii)]{Villani} upto a subsequence  $(N_k, h_k, o_k)$ converges to a complete non-collapsed geodesic space  $(X_\infty, d_\infty, o_\infty)$ of Hausdorff dimension $n$. By Cheeger–Colding  volume convergence theorem \cite[Theorem 5.9]{CH}, the volume measures $\mathrm{Vol}_{h_k}$ converge weakly to the $n$-dimensional Hausdorff measure $\mathcal{H}^n$ on $X_\infty$, and for every $r > 0$ 
$$\lim_{k \to \infty} \mathrm{Vol}_{h_k}(B_r(o_k)) = \mathcal{H}^n(B_r(o_\infty)) \ge \theta_{*} \omega_n r^n.$$
Thus,  $\mathrm{AVR}(X_\infty, \mathcal{H}^n) \ge \theta_{*},$ and for any fixed radius $R > 0,$ 
$$\lim_{k \to \infty} \int_{B_{o_k}(R)} e^{-d_k(o_k, x)^2} \, dV_{h_k} = \int_{B_{o_\infty}(R)} e^{-d_\infty(o_\infty, x)^2} \, d\mathcal{H}^n.$$ Moreover, as  $\mathrm{Vol}_{h_k}(B_{o_k}(r)) \le \omega_n r^n$ for all $k$ and all $r > 0$, co-area formula yields 
$$\int_{N_k \setminus B_{o_k}(R)} e^{-d_k(o_k, x)^2} \, dV_{h_k} \le \int_R^\infty 2 \omega_n r^{n+1} e^{-r^2} \, dr.$$ Because this bound is independent of $k$ and vanishes as $R \to \infty$, the tail integrals over $N_k \setminus B_{o_k}(R)$ (and likewise over $X_\infty \setminus B_{o_\infty}(R)$) decay uniformly to $0$. Hence,

\begin{align*}
\lim_{k \to \infty} \int_{N_k}e^{-d_k(o_k, x)^2} \, dV_{h_k}=\int_{X_\infty} e^{-d_\infty(o_\infty, x)^2} \, d\mathcal{H}^n  > 0.
\end{align*}

Similarly, $$\lim_{k \to \infty} \int_{N_k} d_k(o_k, x)^2 e^{- d_k(p_k, x)^2} dV_{h_k} = \int_{X_\infty} d_{o_\infty}(o_\infty, x)^2 e^{- d_{o_\infty}(o_\infty, x)^2} > 0.$$
This leads to $\lim_{k\to\infty}\tilde I( u_1)_{N_k, h_k}>0$ which is a contradiction.
Hence $m_0(n, \theta_{*}) > 0.$ Thus Step 5 follows. 
Therefore,
\begin{align*}
\delta(M) = \tilde I_M( u_{1})\geq\inf_{\alpha\geq 1}\tilde I_M(u_\alpha) \geq  C\sup_{\alpha\geq 1}\tilde I_M(u_\alpha) \geq C(n,\theta_{*})\mbox{AVR}(M)^{2-1/n}\mbox{dist}(M,\R^n).
\end{align*}
This completes the proof. 
\end{proof}

\vspace{2mm}

\subsection{On the best HPW-constant for $\mbox{Ric} \geq 0$}

Let $(M,g)$ be  a non-compact, Riemannian $n$-manifold with $\mbox{Ric}_M \geq 0.$ We further assume that $M$ is not isometric to the Euclidean space, in other words $\theta: = \mbox{AVR}(M) < 1.$ We set 
\begin{align*}
J(u) = \dfrac{\left( \int_M \vert{}\nabla_g u\vert{}^2 dV_g \right) \left( \int_M d(x,o)^2 u^2 dV_g \right)}{\frac{n^2}{4}\left( \int_M u^2 dV_g \right)^2}.
\end{align*}
Then by the HPW inequality and testing $J(u)$ against a sequence of concentrating sequence of Euclidean Gaussian we know that 
\begin{align*}
\theta^{\frac{2}{n}} \leq \inf_{u \in C_c^{\infty}(M), u \neq 0} J(u) \leq 1.
\end{align*}
It is a natural question, what is the infimum. Of course, one would wonder that the infimum should be $\theta^{\frac{2}{n}}.$ However, thanks to our non-quantitative stability, we prove in the following that the infimum is actually strictly bigger than
$\theta^{\frac{2}{n}},$ unless of course, the manifold is isometric to the Euclidean space.

{\bf Proof of Corollary~\ref{cor:best-const-1}:}

\begin{proof}
The strict inequality on the right follows from the rigidity result of Krist\'aly. Thus, we focus only on the strict inequality on the left. 

Assume, on the contrary that  $\inf_{u \in C_c^{\infty}(M), u \neq 0} J(u) = \theta^{\frac{2}{n}}.$ Then there exists a sequence $\{u_k\}$ such that $\|u_k\|_2 = 1$  and $J(u_k) \rightarrow \theta^{\frac{2}{n}},$ i.e., $$\left( \int_M \vert{}\nabla_g u_k\vert{}^2 dV_g \right) \left( \int_M d(x,o)^2 u_k^2 dV_g \right)\to \frac{n^2}{4}\theta^{2/n}.$$ Hence the HPW-deficit goes to zero. Furthermore, up to a subsequence one of the following three alternative holds:
\begin{align*}
&(a) \ \lim_k \int_M d(x,o)^2 u_k^2 \ dV_g = 0, \\
&(b) \ \lim_k \int_M d(x,o)^2 u_k^2 \ dV_g = \infty. \\
&(c) \  \int_M d(x,o)^2 u_k^2 \ dV_g \ \ \mbox{remains bounded away from zero and infinity.}
\end{align*}
We set $\lambda_k^2 := \int_M d(x,o)^2 u_k^2 \ dV_g.$ 

If case (c) occurs, then it implies  $\int_M |\nabla u_k|^2 dV\leq C$ for all $k\geq 1$, which in turn implies  up to a subsequence, $u_k \rightharpoonup u$ in $H^1(M)$ for some $u\in H^1(M)$ and therefore, $u_k\to u$ in $L^2_{loc}(M).$ Moreover,
\begin{equation}\label{8-15-1}
    \int_{M\setminus B_R(o)}u_k^2dV_g \leq \frac{1}{R^2}\int_{M\setminus B_R(o)}d(o,x)^2u_k^2dV_g\leq\frac{C}{R^2}\to 0 \quad\mbox{as} \quad R\to\infty.
\end{equation}
Consequently, $u_k\to u$ in $L^2(M)$ and $\|u\|_{L^2(M)}=1$. Further, by weak lower semi-continuity 
$$J(u)\leq\lim\inf_{k\to \infty}J(u_k)=\theta^{2/n}=\inf_{w\in C^\infty_c(M)}J(w),$$
which implies the infimum is achieved by $u$. Then using symmetrization, equality will hold in P\'olya-Szeg\"o inequality \cite[Proposition 3.1]{BK} (or, by Theorem \ref{th:stability-3} and subsequent remarks) which implies $\theta=1$ and this yields a contradiction. Hence, $(c)$ can not occur.

\medskip

If case (b) occurs, then we know that $(M, \lambda_k^{-2} g, o)$ converges to the tangent cone at infinity of $M$. On the other hand, Theorem~\ref{th:stability-3} implies upto a subsequence $(M, \lambda_k^{-2} g, o)$   converges to the Eucludean space in the  pointed Gromov-Hausdorff topology. In other words, the tangent cone at infinity is Euclidean, which together with the $\mbox{Ric}_M \geq 0$ forces $M$ to be the Euclidean space by Colding's rigidity result \cite[Theorem 0.3]{Col}. This yields a contradiction and hence case (b) can not occur. 

\medskip

Hence, case (a) holds. Combining this along with the first inequality in \eqref{8-15-1} implies 

$\int_{M\setminus B_o(r)}u_k^2dV_g \to 0$ for all $r>0$. Therefore, 
\begin{equation}\label{8-15-2}
    u_k^2 \, dV_g \xrightarrow{\mbox{\tiny{weak}}^*} \delta_o.
\end{equation}
From here we claim that $\lim_{k\to\infty} J(u_k) \geq 1,$ contradicting the hypothesis. 

To see this, we choose $\delta > 0$ strictly smaller than the injectivity radius of $M$ at the basepoint $o$. Then $\exp_o : B_0(\delta) \subset \mathbb{R}^n \to B_o(\delta) \subset M$ is a diffeomorphism. Define
$$\tilde{u}_k(x) := u_k\big(\exp_o(x)\big) \quad \text{for } x \in B_\delta(0) \subset \mathbb{R}^n.$$
In geodesic normal coordinates centered at $o$ above, the metric tensor is $g_{ij}(x) = \delta_{ij} + O(\vert{}x\vert{}^2)$ and $g^{ij}(x) = \delta_{ij} + O(\vert{}x\vert{}^2)$, and the volume element is $dV_g = \sqrt{\det g(x)} \, dx = (1 + O(\vert{}x\vert{}^2)) \, dx.$
Now we take a cut-off function $\phi \in C_c^\infty(B_0(\delta))$ such that $0 \le \phi \le 1$, $\phi(x) = 1$ for $|x|\le \delta/2$ and 
we define
$w_k(x) :=  \phi(x) \, \tilde{u}_k(x)$ if $x \in B_0(\delta)$ and extend by $0$ to $\R^n$. By (a), $\lambda_k\to 0$. We define a rescaling of $w_k$ by
$$v_k(y) := \lambda_k^{n/2} w_k(\lambda_k y) = \lambda_k^{n/2} \phi(\lambda_k y) \, u_k\big(\exp_o(\lambda_k y)\big) \quad \text{for all } y \in \mathbb{R}^n.$$
Therefore, $\mbox{supp}(v_k)\subseteq B_0(\delta/2)$ for $k\gg 1$. Therefore, using \eqref{8-15-2}, we obtain
\begin{align*}
    &\int_{\mathbb{R}^n} v_k(y)^2 \, dy = \int_M u_k^2 \, dV_g + o(1) = 1 + o(1), \\
   & \int_{\mathbb{R}^n} \vert{}y\vert{}^2 v_k(y)^2 \, dy = \frac{1}{\lambda_k^2} \int_M d(x,o)^2 u_k^2 \, dV_g + o(1) = 1 + o(1),\\
&    \int_{\mathbb{R}^n} \vert{}\nabla_y v_k(y)\vert{}^2 \, dy = \lambda_k^2 \int_M \vert{}\nabla u_k\vert{}^2 \, dV_g + o(1) = \left(\int_M d(x,o)^2 u_k^2 \, dV_g\right)\left(\int_M \vert{}\nabla u_k\vert{}^2 \, dV_g\right) + o(1) .
\end{align*}
This in turn implies $J(u_k)=J_{\R^n}(v_k)\big(1+o(1)\big)$. Since by the Euclidean HPW inequality on $\mathbb{R}^n$, $J_{\R^n}(v_k)\geq 1$ for all $k\geq 1$, we obtain $J(u_k)\geq 1$ as $k\to\infty$. This yields a contradiction as $\lim_{k\to\infty}J(u_k)=\theta^{\frac{2}{n}}<1.$
Hence, $\inf J > \theta^\frac{2}{n}$ and this completes the proof. 

\end{proof}

\section{Optimal HPW inequality for $\text{Ric}\geq 0$}

We first recall a few prerequisites needed to state the result. The following results are taken from the articles by Mantegazza, Mascellani, and Uraltsev \cite{Mantegazza_et_al} and Cavalletti and Mondino \cite{CMLaplacian}.

\medskip

Let $(M, g)$ be a complete Riemannian $n$-manifold, $n \ge 2$, and fix $o \in M$. Set $r(x) := d(o, x)$. We assume throughout that $Ric_M \ge 0$.

\medskip 

\subsection{Polar Coordinates : Manifolds with $\mbox{Ric}_M \geq 0$}

Let $S_o M := \{\theta \in T_o M : |\theta| = 1\},$
be the unit tangent bundle. For $\theta \in S_o M$, consider the radial geodesic $\gamma_\theta(r) = \exp_o(r\theta)$, and define the cut time and the injectivity domain
\begin{align*}
   t_{\mbox{\tiny cut}}(\theta):=\sup\{t>0:\gamma_\theta|_{[0,t]}\text{ is minimizing}\}, \ \ \mbox{ID}(o)=\bigl\{v\in T_oM:\ |v|<t_{\mathrm{cut}}(v/|v|)\bigr\},
\end{align*}
 and let $\mbox{Cut}(o)$ be the cut locus of $o.$ Then $\mbox{Cut}(o)$ is a closed subset of $M$ of measure zero and 
 \begin{align*}
     \exp_o\big|_{\mbox{ID}(o)}\;:\;\mbox{ID}(o)\;\longrightarrow\;M\setminus\bigl(\operatorname{Cut}(o)\bigr)
 \end{align*}
is a diffeomorphism (\cite[Theorem 10.34]{JLee}). On the open set $M\setminus(\{o\}\cup\mbox{Cut}(o))$, we introduce polar coordinates $(r,\theta)$. In these coordinates the metric and volume form take the form $g=dr^2+g_r(\theta),$
where $g_r$ is a smooth family of Riemannian metrics on the unit sphere of directions and the associated Riemannian volume measure is therefore
$$dV_g=J(r,\theta)\,dr\,d\Theta=r^{n-1}\mathcal{A}(r,\theta)\,dr\,d\Theta.$$
The normalized Jacobian $\mathcal A$ is smooth and positive on the pre-cut domain and  $\lim_{r \rightarrow 0} \mathcal{A}(r,\theta) = 1.$ On $M \setminus (\{o\} \cup \mbox{Cut}(o))$, the distance function is smooth and satisfies $|\nabla_g r| = 1$ and on the smooth region 
\begin{align*}
\Delta_g r =   \frac{\partial_r J}{J} = \frac{1}{r^{n-1} \mathcal{A}} \partial_r \left( r^{n-1} \mathcal{A} \right) = \frac{n-1}{r} + \frac{\partial_r \mathcal{A}}{\mathcal{A}}.
\end{align*}
Since $\mbox{Ric}_M \ge 0$, $\Delta r \le \frac{n-1}{r}$ and hence $\partial_r \mathcal{A} \le 0$. Thus, for every fixed $\theta$, $r \mapsto \mathcal{A}(r, \theta)$ is non-increasing, which implies $\mathcal{A}(r, \theta) \le 1.$

More importantly, $r$ is locally semiconcave on $M \setminus \{o\}$ (see \;\cite{Carlo2}). Therefore, its distributional Hessian, and in particular its distributional Laplacian, is a Radon measure. In fact $\nabla_g r$ belongs to the space
$\mbox{SBV}_{\mbox{\tiny{loc}}}(M \setminus \{o\})$ of vector fields with locally special bounded variation, i.e the distributional derivative has no Cantor part.

Mantegazza, Mascellani, and Uraltsev \cite{Mantegazza_et_al} in the Riemannian setting, and Cavalletti--Mondino \cite{CMLaplacian} in a more general metric-measure framework, studied global representation formulas for the distributional Laplacian of the distance function. We recall the relevant decomposition in the present Riemannian setting. Writing $\Delta_g r$ as a Radon measure, we decompose it into its absolutely continuous and singular parts with respect to the Riemannian volume measure $dV_g$:
\begin{align*}
\Delta_g r = \left( \frac{n-1}{r} + \frac{\partial_r \mathcal{A}}{\mathcal{A}} \right) dV_g + d\nu_s,
\end{align*}
where $d\nu_s \le 0$, $\nu_s \perp dV_g$, and $\mbox{supt}(\nu_s) \subseteq \mbox{Cut}(o)$. The sign $\nu_s \le 0$ comes from the local semiconcavity of the distance function 
(see \cite[Corollary 2.11]{Mantegazza_et_al} for a precise statement).

\subsection{The divergence of the vector field $r\nabla_g r$}
Consider the vector field $X := r\nabla_g r$. Since $|\nabla_g r| = 1$ a.e., we have formally
\begin{align*}
\divop_g(r\nabla_g r) = |\nabla_g r|^2 + r\Delta_g r = 1 + r\Delta_g r.
\end{align*}
Substituting the decomposition of $\Delta_g r$, we obtain
\begin{align*}
\divop_g(r\nabla_g r) &= dV_g + r \left[ \left( \frac{n-1}{r} + \frac{\partial_r \mathcal{A}}{\mathcal{A}} \right) dV_g + d\nu_s \right] \\
&= \left( n + r \frac{\partial_r \mathcal{A}}{\mathcal{A}} \right) dV_g + r \, d\nu_s,\\
&= \left( n - r \frac{|\partial_r \mathcal{A}|}{\mathcal{A}} \right) dV_g + r \, d\nu_s = d (\divop_g(r\nabla_g r))_{\mathrm{ac}} + (\divop_g(r\nabla_g r))_s
\end{align*}
where we used $\partial_r \mathcal{A} \le 0$. 
Although $r$ is non-smooth at $o$, $o$ produces no additional Dirac mass when $n \ge 2$, because near $o$, $r\nabla_g r \sim x$ in normal coordinates, and $\divop_g(x) \sim n$.

\subsection{Integration by Parts}
Let $u \in C_c^\infty(M)$. Since $X = r\nabla r$ is locally bounded and has distributional divergence given by the Radon measure above, the weak divergence theorem gives
\begin{align*}
\int_M \langle \nabla_g(u^2), r\nabla_g r \rangle \, dV_g = -\int_M u^2 \, d(\divop_g(r\nabla_g r)).
\end{align*}
Since $\nabla_g(u^2) = 2u\nabla_g u$, we get
\begin{align*}
-2 \int_M u \langle \nabla u, r\nabla_g r \rangle \, dV_g = \int_M u^2 \, d(\divop_g(r\nabla_g r)).
\end{align*}
Using the decomposition into absolutely continuous and singular parts:
\begin{align}\label{ibp}
-2 \int_M u \langle \nabla_g u, r\nabla_g r \rangle \, dV_g = \int_M u^2 \left( n - r \frac{|\partial_r \mathcal{A}|}{\mathcal{A}} \right) dV_g + \int_M u^2 r \, d\nu_s.
\end{align}

\subsection{HPW for $\mbox{Ric} \geq 0$}

For $u \in C_c^{\infty}(M)$ satisfying $\|u\|_{L^2}^2 = 1$, we set 
\begin{align*}
\mathfrak{D}(u) := \frac{1}{n}\left[\int_M u^2 \left( r \frac{|\partial_r \mathcal{A}|}{\mathcal{A}} \right) dV_g - \int_M u^2 r \, d\nu_s\right].
\end{align*}
 For general $u \in C_c^{\infty}(M)$, simply normalize by it's $L^2$-norm. Then we have from \eqref{ibp}, and Cauchy-Schwartz inequality, for every $u \in C_c^{\infty}(M)$ satisfying $\|u\|_{L^2}^2 = 1$,

\begin{align}
\frac{n^2}{4} \left( 1 - \mathfrak{D}(u) \right)^2 = 
\left(\int_M u \langle \nabla u, r\nabla_g r \rangle \, dV_g\right)^2
\leq  \left(\int_M   |\nabla_g u|^2 \, dV_g\right)\left(\int_M r^2 u^2  \, dV_g\right).
\end{align}

This is the general HPW-uncertainty principle for $\mbox{Ric} \geq 0$ manifolds. Using Lemma \ref{whatis} below, we can extend the above inequality to radially symmetric, decreasing functions without assuming compact support.
\medskip

{\bf Equality cases:} Equality holds if and only if equality holds in Cauchy-Schwartz inequality. Hence there exist a constant $2\alpha > 0$ such that:

$$-\nabla_g u = 2\alpha  u  r \nabla_g r$$

Since $r(x)$ is the distance from the origin, its gradient $\nabla r$ is just the unit radial vector field $\partial_r$. Therefore, taking inner product with $\nabla_g r$ we get
$$-\frac{\partial u}{\partial r} = 2\alpha  ru.$$

Solving the equation we get:

$$u(x) = C e^{-\alpha r(x)^2},$$
 where $C$ is a constant.

  \subsection{What is $\sup_{\|u\|_{L^2} = 1} \mathfrak{D}(u)$} Since the measure $\nu_s$ in the definition of $\mathfrak{D}(u)$ is singular to $dV_g,$ one can not expect  $\sup_{\|u\|_{L^2} = 1} \mathfrak{D}(u)$ to be finite. However, we have the following lemma in this regard. Recall that a function $u : M \to \mathbb{R}$ is said to be radial if $u(x) = \phi (d(o,x))$
  for some function $\phi.$ 
  
  \begin{lemma}\label{whatis}
  We have
   \begin{align*}
 1 - \sup_{\substack{\|u\|_{L^2} = 1 \\ u \text{ radial, decreasing}}} \mathfrak{D}(u) \geq \mbox{AVR(M)}^{\frac{1}{n}}.
  \end{align*}
  \end{lemma}


\begin{proof}

Denote $\theta := \text{AVR}(M) > 0.$
 For a complete, noncompact $n$-dimensional manifold $M$ with $Ric_M \ge 0$ and strictly positive AVR, we have the  following sharp isoperimetric inequality (see \cite[(1.2)]{BK}). For sets of finite perimeter $E \subset M$, 

\begin{align*}
\text{Area}(\partial E) \ge n(\omega_n \theta)^{\frac{1}{n}} \text{Vol}(E)^\frac{n-1}{n}.
\end{align*}

In particular, applying this isoperimetric inequality to the geodesic ball $E = B_o(R)$ and writing $V(R) = \text{Vol}_g(B_o(R))$ and $A(R) = \text{Area}_g(\partial B_o(R))$, we obtain:
\begin{align} \label{iso-1} 
A(R) \ge n(\omega_n \theta)^{\frac{1}{n}} V(R)^{\frac{n-1}{n}}.
\end{align}

By Bishop-Gromov volume comparison $V(R) \le \omega_n R^n.$ Multiply \eqref{iso-1} by $R$ and and using the last bound we get

\begin{align} \label{iso-2}
R A(R) \ge n(\omega_n \theta)^{\frac{1}{n}} R V(R)^{\frac{n-1}{n}} \ge n(\omega_n \theta)^{\frac{1}{n}} \left( \frac{V(R)}{\omega_n} \right)^{\frac{1}{n}} V(R)^{\frac{n-1}{n}} = n \theta^{\frac{1}{n}} V(R).
\end{align}

We denote by
\begin{align}\label{8-17-1}
d\mu := \left( r \frac{|\partial_r \mathcal{A}|}{\mathcal{A}} \right) dV_g - r  d\nu_s = n  dV_g - \divop_g(r\nabla_g r),
\end{align}
where we used $\divop_g(r\nabla_g r) = n \, dV_g - d\mu.$ Therefore, using the definition of $\mathfrak{D}(u)$ we can write
$$\mathfrak{D}(u)=\frac{1}{n}\int_M u^2d\mu.$$
Thus, to complete the proof of this lemma we are left to show 
$$1-\theta^\frac{1}{n}\geq \frac{1}{n}\int_M u^2d\mu$$
for all $u\in C_c^\infty(M)$ with $\|u\|_{L^2(M)}=1$ and $u$ is radial and radially decreasing. 
Applying the divergence theorem and using \eqref{8-17-1}, we have
\begin{align*}
\mu(B_o(R)) = \int_{B_o(R)} n  dV_g - \int_{\partial B_o(R)} \langle r\nabla_g r, \nu \rangle  dA = n V(R) - R A(R).
\end{align*}
(Here we used that on $\partial B_o(R)$, the outward unit normal is $\nu = \nabla_g r$ by Gauss lemma, and so $\langle r\nabla_g r, \nu \rangle = R$.) Substituting the lower bound obtained in \eqref{iso-2}, we get  

\begin{align*}
\mu(B_o(R)) \le n V(R) - n \theta^{\frac{1}{n}} V(R) = n \left( 1 - \theta^{\frac{1}{n}} \right) V(R).
\end{align*}

Thus, for radial and radially decreasing function $u$ with $\|u\|_{L^2(M)}=1$, we can integrate with respect to $\mu$  by layer cake representation and obtain
\begin{align*}
\int_M u^2  d\mu &= \int_0^\infty \mu(\{ u^2 > t \}) \ dt \\
&\le \int_0^\infty n \left( 1 - \theta^{\frac{1}{n}} \right) \text{Vol}_g(\{ u^2 > t \})  \ dt \\
&= n \left( 1 - \theta^{\frac{1}{n}} \right) \int_M u^2  \ dV_g\\
&=n \left( 1 - \theta^{\frac{1}{n}} \right).
\end{align*}
In this step we use that the superlevel sets $\{ u^2 > t \}$ are geodesic balls.
This completes the proof of the lemma. 
\end{proof}


\section{Appendix}
In this appendix, we provide the remaining details of the proof of the rigidity result, Proposition \ref{rigidity thm}. These details are included both to make the argument self-contained and to highlight the geometric ingredients underlying the main result. The following lemma is classical; a proof can be found in \cite[Lemma 12.2]{Gray}.

\begin{lemma}\label{laplace-0}

Let $(M,g)$ be a Cartan Hadamard manifold and $\rho(x)=d(x,x_0)$. Then
    
$$\Delta \rho = \frac{n-1}{\rho} - \frac{Ric(\partial_\rho, \partial_\rho)}{3}\rho + O(\rho^2) \quad\mbox{as}\quad \rho \to 0.$$
\end{lemma}

\begin{lemma}\label{Laplace rigidity lemma}
Let $(M,g)$ be a Cartan-Hadamard manifold with sectional curvature bounded above by $c < 0$ and
\begin{align*}
 \Delta_g d_{x_0} = (n-1)\mathbf{ct}_c(d_{x_0})
\end{align*}
holds in distributional sense. Then 
$(M,g)$ is isometric to $\hn_c$.
\end{lemma}

\begin{proof} 
We divide this proof into three steps. 

Step 1: 
In geodesic polar coordinated, the metric is represented by $g = d\rho^2 + g_{\rho}$. By Hessian comparison 
\begin{align}\label{Hess ineq}
\mbox{Hess} \, d_{x_0}(x) \geq \sqrt{|c|}\,\coth \ \!\bigl(\sqrt{|c|}\,d_{x_0}(x)\bigr) \, g_{\rho}.
\end{align}
Since trace of $\mbox{Hess} \, d_{x_0}(x) = \Delta_g d_{x_0}(x)$, comparing \eqref{Hess ineq} with \eqref{Laplace equality} yields
\begin{equation}\label{Hess-eq}
\mbox{Hess} \, d_{x_0} (x)= \sqrt{|c|}\,\coth\!\bigl(\sqrt{|c|}\,d_{x_0} (x)\bigr) \, g_{\rho}.
\end{equation}

\medskip

{Step 2:} In this step we show that the sectional curvature of $(M,g)$ in the radial direction is identically equal to $c$. The proof is achieved by the comparison of Riccati equation.  

Let $\rho = d_{x_0}(x)$ and $X$ be a parallel field to $\partial_{\rho}$, i.e. $\nabla_{\partial_{\rho}} X = 0.$ Let $v\in T_{x_0}M$ with $\|v\|=1$ and $\gamma(t):=\exp_{x_0}(tv)$ be a unit speed geodesic. Thus, along the ray $\gamma$: 
$$\rho(\gamma(t))=d_{x_0}(\gamma(t))=t.$$
Therefore, $\frac{d}{dt}\|X(t)\|^2=\frac{d}{dt}g(X, X)=2g(\nabla_{\partial_\rho}X,X)=0$, i.e., $\|X(t)\|=\mbox{constant}$ and  $$\frac{d}{dt}g(X, \partial_{\rho}) = g(\nabla_{\partial_\rho}X, \nabla_{\partial_\rho}\partial_\rho)=0,$$ i.e., $g(X, \partial_{\rho})=\mbox{constant}$.  Therefore, we can assume that $g(X, \partial_{\rho}) = 0$ and $g(X, X) = 1$ by first projecting $X$ onto ${\partial_\rho}^\perp$ and then normalizing it. Next  applying Riccati equation on $X$ (see \cite[page 101]{Peter}) yields


\begin{equation}\label{Riccati}
\partial_\rho \big( \operatorname{Hess}\, \rho (X, X) \big)
+ \operatorname{Hess}^2 \rho (X, X)
= -\operatorname{sec}\,(X, \partial_{\rho}).
\end{equation}
where for any two tangent vectors $X,Y$, $\operatorname{sec} \, (X, Y)$ is the sectional curvature of the plane spanned by $X$ and $Y$ and is given by 
\begin{align*}
\operatorname{sec} \, (X, Y) = \frac{g(R(X,Y)Y,X)}{g(X,X)g(Y,Y) - g(X,Y)^2}.
\end{align*}
Here $R(X,Y)$ denotes the Riemann curvature operator and defined by 
$$R(X,Y)Z:=\nabla_X\nabla_Y Z-\nabla_Y\nabla_X Z-\nabla_{[X,Y]}Z.$$
Define
\begin{align*}
U(\rho) := \mbox{Hess} \, \rho (X,X)=\mbox{Hess} \, d_{x_0}(x) (X,X),
  \ \tilde{U}(\rho) := \sqrt{|c|}\,\coth \ \!\bigl(\sqrt{|c|}\,\rho\bigr)
\end{align*}
Since, $\tilde{U}$ solves $\tilde{U}^{\prime} + \tilde{U}^2 = - c$, subtracting the Riccati equation corresponding to $\mathbb{H}^n_c$ from \eqref{Riccati} yields
$$(U'-\tilde{U}^\prime)+(U^2-\tilde{U}^2)=-(\operatorname{sec}\,(X, \partial_{\rho})-c).$$
Consequently, $\operatorname{sec}\,(X, \partial_{\rho})=c$ for every parallel field $X$ to $\partial_{\rho}$  as $U=\tilde{U}$ by \eqref{Hess-eq}. 
That is the sectional curvature in the radial direction is identically equal to $c.$ This proves our claim. 

\medskip

{ Step 3:} In this step, we show that the volume of every geodesic ball centered at $x_0$ coincides with that of the corresponding geodesic ball in the simply connected space form of constant sectional curvature.
The metric in polar coordinates is given by
\[
 d\rho^2 + g_{\rho} = d\rho^2 + g_{ij}(\rho,\theta)\, d\theta^i d\theta^j .
\]
Recall, $|g|: = \det(g_{ij}).$ Then the Laplacian on $M$ in polar co-ordinates takes the form
\begin{align*}
\Delta_g &= \frac{1}{\sqrt{|g|}} \frac{\partial}{\partial \rho} \left( \sqrt{|g|} \frac{\partial}{\partial \rho} \right) + \Delta_{\partial B(0,\rho)},\\
& = \frac{\partial^2}{\partial \rho^2} + m(\rho,\theta)\frac{\partial}{\partial \rho} + \Delta_{\partial B(0,r)},
\end{align*}

where $m(\rho,\theta) = \frac{\partial}{\partial \rho} \left( \ln \sqrt{|g|} \right).$
Since we proved that the radial sectional curvature of $(M,g)$ is equal to $c = -\kappa^2 = -\frac{\psi^{\prime\prime}(\rho)}{\psi(\rho)},$ with $\psi(\rho) = \sinh(\kappa \rho)$, using Lemma~\ref{vol bound}, we obtain
\[
m(\rho,\theta) \geq (n-1)\frac{\psi^{\prime}(\rho)}{\psi(\rho)} = (n-1)\kappa \coth(\kappa\rho) \quad\forall \rho>0, \, \theta\in \mathbb{S}^{n-1}.
\]

Now, applying the Laplacian to the distance function $\rho = d(x_0,x)$, we obtain
\[
\Delta_g \rho = m(\rho,\theta) \geq (n-1)\kappa \coth(\kappa \rho).
\]

But from \eqref{Laplace equality} we have $\Delta_g \rho = (n-1) \kappa \coth(\kappa\rho)$ and hence
\begin{equation}\label{23-7-2}
m(\rho,\theta) = (n-1) \kappa \coth(\kappa\rho).
\end{equation}
Now using the definition of  $m(\rho,\theta)$ i.e.,  $m(\rho,\theta)= \frac{\partial}{\partial\rho} \ln\big( \sqrt{|g|(\rho,\theta)}\big)$ into \eqref{23-7-2}, we obtain
\begin{align*}
\ln \big(\sqrt{|g|(\rho,\theta)} \big)= \int_0^\rho m(s,\theta)\, ds + C(\theta)=(n-1)\ln(\sinh(k\rho))+ C(\theta).
\end{align*}

Therefore,
\[
\sqrt{|g|(\rho,\theta)} = \varphi(\theta)\, \sinh^{\,n-1}(\kappa\rho),
\]
where $ \varphi(\theta)=e^{C(\theta)}$. Consequently, we deduce
\begin{align}\label{24-7-1}
\mathrm{Vol}(B_{x_0}(r)) 
= \int_0^r \int_{\mathbb{S}^{n-1}} \phi(\theta)\, \sinh^{n-1}(\kappa\rho)\, d\theta\, d\rho.=C(n)\int_0^r\sinh^{n-1}(\kappa\rho)\,  d\rho \quad \forall\, r>0.
\end{align}
Using the Taylor expansion:
$\sinh(\kappa\rho) = \kappa\rho + \frac{(\kappa\rho)^3}{6} + O(\rho^5),$
 and $\mathrm{Vol}(B_{x_0}(r))=\omega_n r^n + o(r)^n$ as $r\to 0$, we get from \eqref{24-7-1} 
$$\omega_nr^n=C(n)\int_{0}^r(\kappa\rho)^{n-1}\, d\rho + o(r^n)=c(n)\frac{\kappa^{n-1}r^n}{n} + o(r^n).$$
This in turn implies $C(n)=\frac{n\omega_n}{\kappa^{n-1}}.$ Therefore, from \eqref{24-7-1} we have
\begin{equation}\label{24-7-3}
\mathrm{Vol}(B_{x_0}(r)) =\frac{n\omega_n}{\kappa^{n-1}}\int_0^r\sinh^{n-1}(\kappa\rho)\,  d\rho=\mathrm{Vol}_{\mathbb{H}^n_c}(B_{x_0}(r))  \quad \forall\, r>0.\end{equation}

Now we aim to show that the above volume formula holds for any $x \in M.$ By volume comparison theorem for any $x \in M$ it holds

$$
\mathrm{Vol}_{\hn_c}(B_{x}(r))\leq \mathrm{Vol}(B_{x}(r))  \quad \forall \ r > 0. 
$$

Therefore using the fact that 
$B(x, r) \subset B(x_0, r + d(x, x_0))$ and the above volume formula \eqref{24-7-3} for a fixed pole $x_0 \in M$ we have

\begin{align*}
n\omega_n
&\leq \limsup_{r \rightarrow \infty}
\dfrac{\mathrm{Vol}(B_{x}(r))}{\displaystyle\int_{0}^{r} \left( \frac{\sinh(\kappa s)}{\kappa} \right)^{n-1} \; d\rho} \\
 &\leq \limsup_{r \rightarrow \infty}
\dfrac{\mathrm{Vol}(B_{x_0}(r + d(x, x_0)))}{\displaystyle\int_{0}^{r + d(x, x_0)} \left( \frac{\sinh(\kappa s)}{\kappa} \right)^{n-1} \; d\rho} \cdot\dfrac{\displaystyle\int_{0}^{r + d(x, x_0)} \left( \frac{\sinh(\kappa s)}{\kappa} \right)^{n-1} \; d\rho}{\displaystyle\int_{0}^{r} \left( \frac{\sinh(\kappa s)}{\kappa} \right)^{n-1} \; d\rho} \\
& = n\omega_n.
\end{align*}

Consequently,

\begin{align*}
\mathrm{Vol}(B_{x}(r)) 
= \frac{\omega_{n-1}}{\kappa^{n-1}} \int_0^r \sinh^{n-1}(\kappa\rho)\, d\rho, =\mathrm{Vol}_{\hn_c}(B_{x}(r))\quad \forall x\in M,\, \forall\, r > 0.
\end{align*}

Therefore, by the volume comparison theorem  $B_x(r)\subset M$ is isometric to the ball of radius $r$ in 
$\hn_c$ for all $x\in M$ and $r>0$. Hence, applying \cite[Proposition 4.14(2)]{BF} we obtain a global isometry between $(M,g)$ and $\mathbb{H}^n_{c}$ (since universal cover of any Cartan Hadamard manifold is the manifold itself).
\end{proof}

\vspace{4mm}

\noindent \textbf{Funding:}
The research of M.~Bhakta is supported by the ANRF-MATRICS grant\\
(ANRF/ARGM/2025/000126/MTRDST). The research of D.~Ganguly is partially supported by the ANRF MATRICS Research Grant (MTR/2023/000331) and ANRF-Advanced Research Grant (ANRF/ARG/2025/000572/MS). D. Karmakar acknowledges the support of the Department of Atomic
Energy, Government of India, under Project Identification No. RTI 4014 and ANRF ARG-MATRICS grant ANRF/ARGM/2025/000976/MTR.

\vspace{0.1cm}

\noindent \textbf{Competing interests:} The authors have no competing interests to declare that are relevant to the content of this article.

\vspace{0.1cm} 

\noindent\textbf{Data availability statement:} Data sharing is not applicable to this article as no data sets were generated or analysed during the current study.

\end{document}